\documentclass{IEEEtran} 

\usepackage{graphics} % for pdf, bitmapped graphics files
\usepackage{epsfig} % for postscript graphics files
\usepackage{amsmath} % assumes amsmath package installed
\usepackage{amssymb}  % assumes amsmath package installed
\usepackage{comment}
\usepackage{algpseudocode}
\usepackage{algorithm}

\usepackage{epstopdf}
\usepackage{subfigure}

\usepackage{hyperref}

\usepackage{amsthm}
\usepackage{mathtools}
\usepackage{aligned-overset}

\newtheorem{theorem}{\bf Theorem}

\newtheorem{lemma}{\bf Lemma}

\newtheorem{remark}{\bf Remark}

\usepackage{setspace}
\usepackage{tikz}               
\usepgflibrary{arrows}

\usepackage{enumitem}
 
\newcommand{\R}{{\mathbb R}} % Real numbers

\begin{document}
\title{Robust stability analysis of a simple data-driven model predictive control approach}
\author{Joscha Bongard$^1$, Julian Berberich$^2$, Johannes K\"ohler$^{2,3}$, and Frank Allg\"ower$^2$% <-this % stops a space
\thanks{Funded by Deutsche Forschungsgemeinschaft (DFG, German Research Foundation) under Germany's Excellence Strategy - EXC 2075 - 390740016; grant 468094890; grant AL 316/12-2 - 279734922; and the International Research Training Group Soft Tissue Robotics (GRK 2198/1 - 277536708).
We acknowledge the support by the Stuttgart Center for Simulation Science (SimTech).
The authors thank the International Max Planck Research School for Intelligent Systems (IMPRS-IS) for supporting Julian Berberich.}
\thanks{$^1$Technical University of Munich, Chair of Automatic Control, 85748 Garching (Munich), Germany (e-mail: joscha.bongard@tum.de)}
\thanks{$^2$University of Stuttgart, Institute for Systems Theory and Automatic Control, 70550 Stuttgart, Germany.
(email:$\{$julian.berberich, frank.allgower\}@ist.uni-stuttgart.de)}
\thanks{$^3$Institute for Dynamical Systems and Control, ETH Zurich, ZH-8092, Switzerland (email:jkoehle@ethz.ch)}
}

\IEEEoverridecommandlockouts

\IEEEpubid{\begin{minipage}{\textwidth}\ \\[12pt] \\ \\
\copyright 2021 IEEE. Personal use of this material is permitted. Permission from IEEE must be obtained for all other uses, in any current or future media, including reprinting/republishing this material for advertising or promotional purposes, creating new collective works, for resale or redistribution to servers or lists, or reuse of any copyrighted component of this work in other works.
\end{minipage}}

\maketitle
%
%%%%%%%%%%%%%%%%%%%%%%%%%%%%%%%%%%%%%%%%%%%%%%%%%%%%%%%%%%%%%%%%%%%%%%%%%%%%%%%%
\begin{abstract}
In this paper, we provide a theoretical analysis of closed-loop properties of a simple data-driven model predictive control (MPC) scheme.
The formulation does not involve any terminal ingredients, thus allowing for a simple implementation without (potential) feasibility issues.
The proposed approach relies on an implicit description of linear time-invariant systems based on behavioral systems theory, which only requires one input-output trajectory of an unknown system.
For the nominal case with noise-free data, we prove that the data-driven MPC scheme ensures exponential stability for the closed loop if the prediction horizon is sufficiently long.
Moreover, we analyze the robust data-driven MPC scheme for noisy output measurements for which we prove closed-loop practical exponential stability.
The advantages of the presented approach are illustrated with a numerical example.
\end{abstract}


%!TEX root = ./Main.tex
%%%%%%%%%%%%%%%%%%%%%%%%%%%%%%%%%%%%%%%%%%%%%%%%%%%%%%%%%%%%%%%%%%%%%%%%%%%%%%%
\section{Introduction}
Model predictive control (MPC) is a powerful modern control technique which relies on repeatedly solving an open-loop optimal control problem~\cite{rawlings2017model}.
Key advantages of MPC if compared to other control methods are its applicability to nonlinear systems, the possibility to include constraints on system variables, and desirable closed-loop guarantees on stability and performance.
For the implementation of MPC, an accurate prediction model is required in order to optimize over possible future system trajectories.
In practice, obtaining an accurate model is often time-consuming and requires expert knowledge which explains the increasing interest in designing controllers directly from data without any model knowledge~\cite{hou2013model}.
However, while data-driven control methods are usually simple as no model knowledge is required for their implementation, they often lack the strong theoretical guarantees associated with model-based approaches which are particularly relevant in safety-critical systems.
To this end, we provide a novel theoretical analysis of closed-loop properties of a simple MPC scheme which uses one noisy input-output data trajectory and no model knowledge for prediction, proving that the scheme guarantees stability and robustness under reasonable assumptions.

\vspace{-12pt}
\subsection*{Related work}
In the behavioral approach to control, is was shown that one input-output trajectory of a linear time-invariant (LTI) system can be used to reproduce any further data trajectory of the same system, provided that the input component is persistently exciting~\cite{Willems05}.
This work has received increasing attention for the development of purely data-driven system analysis and control methods, see, e.g.,~\cite{markovsky2008data,romer2019one,DePersis19,berberich2020design}.
Of particular interest is the idea to use the main result in~\cite{Willems05} to develop a data-driven MPC scheme~\cite{yang2015data,DeePC1}.
Various recent works have explored different extensions and modifications of this idea to improve robustness of the algorithm, reduce the computational complexity, and enhance its practical applicability~\cite{fabiani2020optimal,alpago2020extended,baros2020online,yin2020maximum}.
Further, guarantees on robustness and constraint satisfaction for the corresponding open-loop data-driven optimal control problem are provided in~\cite{DeePC4}.
A first analysis of \emph{closed-loop} properties of a simple data-driven MPC scheme is performed in~\cite{Stability&RobustnessTEC}, proving stability and robustness guarantees for noisy output measurements. 
This work is further refined to robust output constraint satisfaction in~\cite{berberich2020constraints} and to a more flexible and practically applicable tracking formulation in~\cite{berberich2020tracking}.
In these works, closed-loop stability is enforced by including a \emph{terminal equality constraint} in the open-loop optimal control problem that is solved online, which is a well-studied idea from model-based MPC~\cite{rawlings2017model}. 
However, since terminal equality constraints are a strong restriction on the online optimization variables, they can deteriorate robustness properties if compared to MPC without terminal ingredients and the theoretical analysis requires the application of the scheme in a multi-step fashion.
Alternatively, a data-driven MPC scheme based on general terminal ingredients, i.e., terminal cost and terminal set constraint, is proposed in~\cite{berberich2021on} which leads to better closed-loop properties, but the computation of these ingredients complicates the design and is not always applicable.

\vspace{-12pt}
\subsection*{Contribution}
In this paper, we provide a closed-loop analysis of a simple data-driven MPC scheme without any stabilizing terminal ingredients.
First, for the case of noise-free data, we apply results on model-based MPC with positive semidefinite cost~\cite{ImplicitSolutions,ForWantOfALocalCLF} to show that, for a sufficiently long prediction horizon, the closed loop is exponentially stable.
In the more\\

\noindent realistic scenario of noisy data, we consider a scheme similar to the one in~\cite{DeePC1,DeePC4} and prove that, under suitable assumptions, the closed loop is practically exponentially stable w.r.t. the noise level.
Our theoretical analysis relies on novel continuity and robustness arguments of MPC based on the data-driven model of~\cite{Willems05}, whose inaccuracy in the presence of noisy data poses a key challenge.
%Our theoretical analysis extends the proof idea in the robust model-based MPC~\cite{NovelConstraintTightening} to the data-driven model of~\cite{Willems05} which is inaccurate in the presence of noisy data.
Since the presented approach does not rely on terminal equality constraints, it has multiple advantages over the existing closed-loop robustness guarantees for data-driven MPC provided in~\cite{Stability&RobustnessTEC} such as improved robustness, better numerical properties, and the application in a one-step MPC scheme.
We illustrate these advantages in a numerical example.

\subsection*{Outline}
The remainder of the paper is structured as follows.
In Section~\ref{sec:prelim}, we introduce some preliminaries regarding extended state-space systems and the data-driven system parametrization based on the Fundamental Lemma.
In Section~\ref{sec:nom}, we provide a data-driven MPC scheme for noise-free data with guaranteed closed-loop stability.
Section~\ref{sec:robust} addresses the problem of robust data-driven MPC for noisy data, where we establish practical exponential stability of the closed loop.
Finally, we apply the MPC to a numerical example in Section~\ref{sec:num} and we conclude the paper in Section~\ref{sec:conclusion}.

\subsection*{Notation}
We write $I_n$ for an $n\times n$-identity matrix and $0_{n\times m}$ for an $n\times m$-zero matrix.
The Euclidean, $\ell_1$-, and $\ell_{\infty}$-norm of a vector $x$ is denoted by $\lVert x\rVert_2$, $\lVert x\rVert_1$, and $\lVert x\rVert_{\infty}$, respectively.
Moreover, the quadratic norm with respect to a positive semidefinite matrix $Q=Q^\top$ is denoted by $\|x\|_Q^2=x^\top Q x$ and the minimum and maximum eigenvalue of $Q$ are denoted by $\lambda_{\min}(Q)$ and $\lambda_{\max}(Q)$, respectively. 
We denote the %non-negative real numbers by $\mathbb{R}_{\geq 0} = \{ r\in\mathbb{R}|r\geq 0\}$, the positive real numbers by $\mathbb{R}_{> 0} = \{ r\in\mathbb{R}|r>  0\}$, and the 
integers in the interval $[a,b]$ by $\mathbb{I}_{[a,b]}$.
By $\mathbb{B}_\delta$, we denote a ball of radius $\delta$, i.e., $\mathbb{B}_\delta \coloneqq \left\{ x \in \R^n | \| x \|_2 \leq \delta \right\}$.
By $\mathcal{K}_{\infty}$ we denote the class of functions $\alpha:\mathbb{R}_{\geq 0}\rightarrow\mathbb{R}_{\geq 0}$, which are continuous, strictly increasing, unbounded, and satisfy $\alpha(0)=0$.
%By $\mathcal{K}$ we denote the class of functions $\alpha:\mathbb{R}_{\geq 0}\rightarrow\mathbb{R}_{\geq 0}$, which are continuous, strictly increasing and satisfy $\alpha(0)=0$. 
%By $\mathcal{K}_{\infty}$ we denote the class of functions $\alpha\in\mathcal{K}$, which are also unbounded.
For a given sequence $\{u_k\}_{k=0}^{N-1}$ and integers $a,b,L$, we define 
\begin{align*}
u_{[a,b]}=\begin{bmatrix}u_{a}\\\vdots\\u_{b}\end{bmatrix}
\end{align*}
as well as the Hankel matrix
\begin{align*}
H_L(u)=\begin{bmatrix}u_0&u_1&\dots&u_{N-L}\\
u_1&u_2&\dots&u_{N-L+1}\\
\vdots&\ddots&\ddots&\vdots\\
u_{L-1}&u_L&\dots&u_{N-1}\end{bmatrix}.
\end{align*}

%!TEX root = ./Main.tex
%%%%%%%%%%%%%%%%%%%%%%%%%%%%%%%%%%%%%%%%%%%%%%%%%%%%%%%%%%%%%%%%%%%%%%%%%%%%%%%
\section{Preliminaries}
\label{sec:prelim}
In this section, we present the problem setting and discuss some preliminaries regarding the representation of LTI systems using input-output data.

\subsection{Problem setting}
We consider discrete-time multiple-input multiple-output LTI systems of the form
\begin{align}
	\begin{split}
		x_{t+1} &= Ax_t + Bu_t,\\
		y_t &= Cx_t + Du_t
	\end{split} \label{eq:MinimalLTIsystem}
\end{align}
with state $x_t\in\mathbb{R}^n$, control input $u_t\in\mathbb{R}^m$, and output $y_t\in\mathbb{R}^p$.
%\JK{Assumption: $G$ minimal!,i.e., $(A,B)$ controllable and $(A,C)$ detectable}
%We consider general input and output constraints $(u_t,y_t)\in\mathbb{U}\times\mathbb{Y}$ that should be satisfied 
We assume that the model $(A,B,C,D)$ is unknown and only (noisy) input-output measurements are available. 
The control goal is to drive the system to the origin while satisfying input and output constraints $(u_t,y_t)\in\mathbb{U}\times\mathbb{Y}$ for given $\mathbb{U}\subseteq\mathbb{R}^m$, $\mathbb{Y}\subseteq\mathbb{R}^p$.

Throughout this paper, we make the standing assumption %assume 
that the matrices in~\eqref{eq:MinimalLTIsystem} are a minimal realization, i.e., $(A,B)$ is controllable and $(A,C)$ is observable.
We consider the problem of stabilizing the origin using a quadratic stage cost with positive definite weighting matrices $Q\in\mathbb{R}^{p\times p}$ and $R\in\mathbb{R}^{m\times m}$ and assume that the origin is in the interior of the constraints, i.e., $0\in\text{int}(\mathbb{U}\times\mathbb{Y})$.

\subsection{Extended state-space system}%Input-output system representation}
In the following, we briefly recap some basics regarding
an equivalent representation of the system~\eqref{eq:MinimalLTIsystem} using an extended state/autoregressive (ARX) model.
Using that $(A,C)$ is observable, the lag $\underline{l}\leq n$ of the LTI system~\eqref{eq:MinimalLTIsystem} is defined as the smallest integer $l \in \mathbb{N}$ such that the observability matrix 
	\begin{align}
	\label{eq:obs_matrix}
		\mathcal{O}_{l} \coloneqq \begin{bmatrix}
			C\\
			CA\\
			\vdots\\
			CA^{l-1}
		\end{bmatrix}
\end{align}
%$\mathcal{O}_l$ of $(A,C)$
 has rank $n$. % (cf.~\cite[Def.~3]{ProvablyRobustVerificationOfDissipativity}).
Given an upper bound $l\geq \underline{l}$ on the lag, we define the extended state $\xi$ at time $t$ as
\begin{align} \label{eq:ExtStateDefinition}
	\xi_t = \begin{bmatrix}
		u_{[t-l,t-1]}\\
		y_{[t-l,t-1]}
	\end{bmatrix}\in\mathbb{R}^{n_\xi},\quad n_\xi=l(m+p).
\end{align}

The dynamics of this extended state are given by a (typically non-minimal) LTI sytem
\begin{align}
	\begin{split}
		\xi_{t+1} = \tilde{A} \xi_t + \tilde{B} u_t,\\
		y_t = \tilde{C} \xi_t + \tilde{D} u_t,
	\end{split} \label{eq:xiDynamics}
\end{align}
%$\xi_t\in\mathbb{R}^{n_\xi}$, $n_\xi:=l(m+p)$. 
which has the same input/output (I/O) behavior as the original LTI system~\eqref{eq:MinimalLTIsystem}~\cite{Goodwin84}.
Furthermore, there exists a (typically non-invertible) matrix $T\in\mathbb{R}^{n\times n_\xi}$ such that $T\xi_t=x_t$ (cf. \cite[Lemma~3]{ProvablyRobustVerificationOfDissipativity}).

%\begin{definition} \label{def:Lag}
%	Given the LTI system $G$, define the lag $\underbar{$l$}$ of $(A,C)$ as the smallest integer $l \in \mathbb{N}$ such that the observability matrix $\mathcal{O}_l$ of $(A,C)$ given by
%	\begin{align}
%		\mathcal{O}_{l} \coloneqq \begin{bmatrix}
%			C\\
%			CA\\
%			\vdots\\
%			CA^{l-1}
%		\end{bmatrix}
%	\end{align}
%	has rank $n$.
%\end{definition}

The following lemma summarizes some relevant properties of the non-minimal representation~\eqref{eq:ExtStateDefinition}--\eqref{eq:xiDynamics}.
\begin{lemma}
\label{lemma:extended}
%Suppose $(A,B)$ is stabilizable and $(A,C)$ is observable with lag $\underline{l}\leq l$.
%Then, 
The system~\eqref{eq:xiDynamics} is such that $(\tilde{A},\tilde{B})$ is stabilizable and $(\tilde{A},\tilde{C})$ is detectable. 
Furthermore, the system~\eqref{eq:ExtStateDefinition}--\eqref{eq:xiDynamics} is input-output-to-state stable (IOSS), i.e., there exists a positive semidefinite matrix $P_{\mathrm{o}}\in\mathbb{R}^{n_\xi\times n_\xi}$ and a constant $\epsilon_{\mathrm{o}}>0$ such that for all $t\in\mathbb{N}$ the trajectories of the system~\eqref{eq:xiDynamics} satisfy
\begin{align}
\label{eq:IOSS}
%\|\xi_{t+1}\|_{P_{\mathrm{o}}}^2\leq \|\xi_t\|_{P_{\mathrm{o}}}^2+\|u_t\|^2+\|y_t\|^2-\epsilon_{\mathrm{o}}\|\xi_t\|^2.
W(\xi_{t+1})\leq W(\xi_t)+\|u_t\|_Q^2+\|y_t\|_R^2-\epsilon_{\mathrm{o}}\|\xi_t\|^2,
\end{align} 
with the IOSS Lyapunov function $W(\xi):=\|\xi\|_{P_{\mathrm{o}}}^2$.
\end{lemma}
\begin{proof}
\textbf{Proof of stabilizability: } 
Using that $(A,B)$ is stabilizable, there exists a matrix $K$ such that $(A+BK)$ is Schur.
Hence, for the input $u_t=Kx_t$, the variables $(x,u,y)$ exponentially converge to zero.
Due to the definition of $\xi$ (cf.~\eqref{eq:ExtStateDefinition}) this implies that $\xi$ exponentially converges to zero and thus $\tilde{A}+\tilde{B}\tilde{K}$ with $\tilde{K}=KT$ is Schur, i.e., $(\tilde{A},\tilde{B})$ is stabilizable. \\
\begin{comment}
\textbf{Part Ib: } Given $0\in\text{int}(\mathbb{U}\times\mathbb{Y})$, there exists a small enough constant $\delta>0$, such that for any initial condition $\xi_0=Tx_0$ with $\|\xi_0\|\leq \epsilon$, the closed-loop with the feedback $u_t=\tilde{K}\xi_t$ satisfies $(u_t,y_t)\in\mathbb{U}\times\mathbb{Y}$. Furthermore, the quadratic cost in combination with exponential stability ensures that the infinite horizon cost can be bounded using some constant $\gamma_{\mathrm{s}}>0$, i.e., $\sum_{t=0}^\infty\|u_t\|_R^2+\|y_t\|_Q^2\leq \gamma_{\mathrm{s}}\|\xi_0\|^2$. 
Thus, for $\|\xi_t\|\leq \delta$ this linear feedback corresponds to a feasible solution in the MPC~\eqref{eq:nominalDDUNCMPC} with a corresponding vector $\alpha(t)$ due to the Lemma~\ref{lem:FundamentalLemma}. 
Correspondingly, we get $J_L^*(\xi_t)\leq \gamma_{\mathrm{s}}\|\xi_t\|^2$. \\
\end{comment}
\textbf{Proof of detectability: } 
Suppose $(u_t,y_t)=0$ for all $t\geq 0$.
This implies $\xi_t=0$ for all $t\geq l$ and thus $(\tilde{A},\tilde{C})$ is detectable, i.e.,~\eqref{eq:IOSS} holds with the quadratic IOSS Lyapunov function $W(\xi)=\|\xi\|_{P_{\mathrm{o}}}^2$ (cf.~\cite[Section 3.2]{Cai08}).
\end{proof}

\subsection{The Willems et al. Fundamental Lemma}
In order to control the system~\eqref{eq:MinimalLTIsystem} without any model identification step, we utilize the following lemma originally proposed in~\cite{Willems05}.
\begin{lemma}[Fundamental Lemma~\cite{Willems05}]\label{lem:FundamentalLemma}
Suppose $\{u_t^\mathrm{d}, y_t^\mathrm{d}\}_{t=0}^{N-1}$ is a trajectory of~\eqref{eq:MinimalLTIsystem}, and $u^\mathrm{d}$ is persistently exciting of order $L+n$. Then, $\{\bar{u}_t, \bar{y}_t\}_{t=0}^{L-1}$ is a trajectory of $G$ if and only if there exists $\alpha \in \mathbb{R}^{N-L+1}$ such that
	\begin{align}
		\begin{bmatrix}
			H_L(u^\mathrm{d})\\
			H_L(y^\mathrm{d})
		\end{bmatrix} \alpha =
		\begin{bmatrix}
			\bar{u}\\
			\bar{y}
		\end{bmatrix}.
	\end{align}
\end{lemma}
This result allows for an equivalent characterization of LTI systems in terms of measured I/O data using the Hankel matrix $H_L$. 
We note that this equivalence is used for many data-driven approaches for LTI systems such as data-driven simulation~\cite{markovsky2008data}, system analysis~\cite{romer2019one}, controller design~\cite{DePersis19,berberich2020design}, and predictive control~\cite{DeePC4,Stability&RobustnessTEC}.
%control and analysis, compare, e.g., \cite{markovsky2008data}, \cite{romer2019one}, \cite{DePersis19,berberich2020design}, \cite{DeePC4,Stability&RobustnessTEC}. 

%!TEX root = ./Main.tex
%%%%%%%%%%%%%%%%%%%%%%%%%%%%%%%%%%%%%%%%%%%%%%%%%%%%%%%%%%%%%%%%%%%%%%%%%%%%%%%
\section{Nominal Data-Driven MPC}
\label{sec:nom}
%In this section, we study the closed-loop properties of a nominal data-driven MPC, i.e., assuming noise-free measurements and data, as a preliminary to the robustness result in Section~\ref{sec:robust}. 

In this section, we study the closed-loop properties of a simple nominal data-driven MPC based on Lemma~\ref{lem:FundamentalLemma}.
We consider the case of noise-free measurements and data, as a preliminary to the robustness result in Section~\ref{sec:robust}.

%\subsection{Nominal Data-Driven MPC Scheme}
Given I/O data $\{ u^\mathrm{d}_k,y^\mathrm{d}_k\}_{k=0}^{N-1}$ and past I/O measurements $(u_{[t-l,t-1]},y_{[t-l,t-1]})$, the nominal data-driven MPC scheme is defined as
\begin{subequations}
	\begin{align}
		J_{L}^{*}&\left(u_{[t-l, t-1]}, y_{[t-l, t-1]}\right) =
		\min _{\alpha(t) \atop \bar{u}(t), \bar{y}(t)} \sum_{k=0}^{L-1} \|\bar{u}_k\|_R^2 + \|\bar{y}_k\|_Q^2 \label{eq:nominalDDUNCMPC_Cost}\\ 
		\text{s.\,t. }&\left[\begin{array}{c}\bar{u}_{[-l, L-1]}(t) \\ \bar{y}_{[-l,L-1]}(t)\end{array}\right]=\left[\begin{array}{c}H_{L+l}(u^\mathrm{d})\\ H_{L+l}(y^\mathrm{d})\end{array}\right]\alpha(t), \label{eq:nominalDDUNCMPC_Dynamics}\\
		& \left[\begin{array}{c}\bar{u}_{[-l, -1]}(t) \\ \bar{y}_{[-l,-1]}(t)\end{array}\right] = \left[\begin{array}{c}u_{[t-l, t-1]}\\ y_{[t-l,t-1]}\end{array}\right], \label{eq:nominalDDUNCMPC_IC}\\ 
		& \bar{u}_{k}(t) \in \mathbb{U},\ \bar{y}_{k}(t) \in \mathbb{Y}\ \forall k \in \mathbb{I}_{[0, L-1]}.
	\end{align}
	\label{eq:nominalDDUNCMPC}
\end{subequations}
We denote an optimal solution to Problem \eqref{eq:nominalDDUNCMPC} by $\alpha^*(t)$, $\bar{u}^*(t)$, $\bar{y}^*(t)$. 
Compared to a model-based MPC, the standard prediction model is replaced by the implicit data-based constraint~\eqref{eq:nominalDDUNCMPC_Dynamics} (cf. Lemma~\ref{lem:FundamentalLemma}) which implies that $(\bar{u}(t),\bar{y}(t))$ is a trajectory of the system~\eqref{eq:MinimalLTIsystem}.
Further, past $l$ I/O measurements are used in~\eqref{eq:nominalDDUNCMPC_IC} to indirectly specify initial conditions for the state $x_t$ (cf. also the extended state $\xi_t$ in Sec.~\ref{sec:prelim}).
% and the initial state $x_t$ is indirectly specified using the past $l$ I/O measurements in~\eqref{eq:nominalDDUNCMPC_IC} (cf. extended state $\xi_t$ in Sec.~\ref{sec:prelim}).
The closed loop is defined using a standard receding horizon implementation as summarized in Algorithm~\ref{alg:nominalMPC}.
\begin{algorithm}[H]
	\caption{Nominal Data-Driven MPC Scheme}
	\label{alg:nominalMPC}
	\begin{algorithmic}[1]
		\Statex \textbf{Given:} horizon length $L$, I/O data $\{u_k^\mathrm{d},y_k^\mathrm{d}\}_{k=0}^{N-1}$, where $u^\mathrm{d}_{[0,N-1]}$ is persistently exciting of order $L+l+n$, the constraint sets $\mathbb{U},\mathbb{Y}$ and I/O weighting matrices $Q,R$.
		\State  At time $t$, take the past $l$ measurements $(u_{[t-l,t-1]}$, $y_{[t-l,t-1]})$ and solve~\eqref{eq:nominalDDUNCMPC}.
		\State Apply the input $u_t=\bar{u}_0^*(t)$.
		\State Set $t \coloneqq t+1$ and go back to 1.
	\end{algorithmic}
\end{algorithm}
For the following theoretical analysis of Algorithm~\ref{alg:nominalMPC}, we use the Lyapunov function 
\begin{align}
Y_L(\xi_t)\coloneqq J_L^*(\xi_t)+W(\xi_t)
\end{align}
with the IOSS Lyapunov function $W$, compare Lemma~\ref{lemma:extended}.
\begin{theorem}[Nominal stability guarantees]
\label{thm:NominalStabilityGuarantees}
Suppose $u^\mathrm{d}$ is persistently exciting of order $L+n+l$.
Then, for any constant $\bar{Y} > 0$, there exist a sufficiently long prediction horizon $L_{\bar{Y}}> 0$, such that for all $L > L_{\bar{Y}}$ and any initial condition satisfying 
$ Y_L(\xi_0)\leq \bar{Y}$, the MPC problem \eqref{eq:nominalDDUNCMPC} is recursively feasible, the constraints are satisfied and the origin is exponentially stable for the resulting closed-loop system.  
\end{theorem}
\begin{proof}
Lemma~\ref{lem:FundamentalLemma} ensures that the data-driven MPC is equivalent to the model-based MPC studied in~\cite{ImplicitSolutions} (compare also~\cite{DeePC1}). 
Inequality~\eqref{eq:IOSS} from Lemma~\ref{lemma:extended} corresponds to the detectability condition~\cite[Ass.~2]{ImplicitSolutions}. 
Using that $0\in\text{int}(\mathbb{U}\times\mathbb{Y})$, there exists a small enough constant $\delta>0$ such that for any initial condition $\xi_0=Tx_0$ with $\|\xi_0\|\leq \delta$,  the closed loop with the stabilizing feedback $u_t=\tilde{K}\xi_t$ and $\tilde{K}$ from Lemma~\ref{lemma:extended} satisfies $(u_t,y_t)\in\mathbb{U}\times\mathbb{Y}$, $\forall t\geq 0$. Furthermore, the quadratic stage cost in combination with an exponential controllability argument ensures that the infinite horizon cost can be bounded using some constant $\gamma_{\mathrm{s}}>0$, i.e., 
\begin{align}\label{eq:thm_nominal_gammas}
J_L^*(\xi_t)\leq \gamma_{\mathrm{s}}\|\xi_t\|^2.
\end{align}
Thus, the local stabilizability condition~\cite[Ass.~1]{ImplicitSolutions} also holds. 
Hence, we can apply~\cite[Thm.~1]{ImplicitSolutions} to conclude that the closed loop satisfies
	\begin{subequations}
	\label{eq:NominalLyapunov}
		\begin{align}
			\epsilon_{\mathrm{o}} \| \xi_t \|_2^2 \leq Y_L(\xi_t) &\leq \gamma_{\bar{Y}} \| \xi_t \|_2^2, \label{eq:NominalLyapunovBounds}\\
			Y_L(\xi_{t+1}) - Y_L(\xi_t) &\leq -\alpha_L \epsilon_{\mathrm{o}} \| \xi_t \|_2^2, \label{eq:NominalLyapunovDecayBounds}
		\end{align}
	\end{subequations}
with some constants $\gamma_{\bar{Y}},\alpha_L>0$. 
 Exponential stability of $\xi=0$ follows from~\eqref{eq:NominalLyapunov} using standard Lyapunov arguments. 
%$\sum_{t=0}^\infty\|u_t\|_R^2+\|y_t\|_Q^2\leq \gamma_{\mathrm{s}}\|\xi_0\|^2$. 
\end{proof}

Theorem~\ref{thm:NominalStabilityGuarantees} provides a lower bound on the prediction horizon $L$ that guarantees the desirable closed-loop properties for the nominal data-driven MPC.   
This lower bound depends on $\bar{Y}$ and hence, on the size of the guaranteed closed-loop region of attraction.
Using the stabilizability and detectability properties derived in Lemma~\ref{lemma:extended}, the result directly follows from model-based MPC theory with semidefinite cost in~\cite{ImplicitSolutions}.
Due to the absence of terminal ingredients, the provided theoretical guarantees only hold for a sufficiently large prediction horizon and choosing a too small prediction horizon $L$ may result in an unstable closed loop, compare the numerical example in~\cite{Stability&RobustnessTEC}.
It is possible to compute the constants in~\eqref{eq:NominalLyapunov} and thus the bound $L_{\bar{Y}}$ on the prediction horizon required for stability explicitly if the system constants $\epsilon_0$ and $\gamma_s$ for detectability and stabilizability are known.
These constants can be determined using only measured data and no model knowledge based on related works on data-driven dissipativity analysis and robust control~\cite{ProvablyRobustVerificationOfDissipativity,berberich2020combining}.
Finally, we note that the same data-driven MPC formulation has also been suggested in~\cite{yang2015data,DeePC1}, however, without a corresponding closed-loop stability analysis.

\begin{remark}
\label{rk:RoA}
	The guaranteed region of attraction (RoA) of the closed loop as stated in Theorem~\ref{thm:NominalStabilityGuarantees} is defined only implicitly via the upper bound on the Lyapunov function $\bar{Y}$. 
         The required prediction horizon $L_{\bar{Y}}$ scales linearly with the value $\bar{Y}$. 
          For $\bar{Y} \rightarrow \infty$, the RoA approaches the set of initially feasible states. 
          The explicit characterization of the RoA is, however, challenging, similar to model-based MPC~\cite{ImplicitSolutions}.
        In the important special case that we only have input constraints and the system is open-loop stable (or similarly if we have no constraints), Theorem~\ref{thm:NominalStabilityGuarantees} can ensure global stability. 
            In this case, the sufficiently large prediction horizon reduces to $L>1+\dfrac{\gamma_{\mathrm{s}}(\gamma_{\mathrm{s}}+\gamma_{\mathrm{o}}-\epsilon_{\mathrm{o}})}{\epsilon_{\mathrm{o}}^2}$ (cf.~\eqref{eq:horizon_bound_stability}).
\end{remark}
%!TEX root = ./Main.tex
%%%%%%%%%%%%%%%%%%%%%%%%%%%%%%%%%%%%%%%%%%%%%%%%%%%%%%%%%%%%%%%%%%%%%%%%%%%%%%%
\section{Robust data-driven MPC}\label{sec:robust}
In practical applications, some measure of noise in offline and online data is unavoidable. 
This causes generally undesirable effects in the nominal MPC scheme \eqref{eq:nominalDDUNCMPC}, e.\,g., deteriorated performance or infeasibility since Lemma~\ref{lem:FundamentalLemma} no longer provides an exact parametrization of the space of system trajectories in \eqref{eq:nominalDDUNCMPC_Dynamics}.
In this section, we consider a modification of the nominal data-driven MPC scheme~\eqref{eq:nominalDDUNCMPC} and we prove that the scheme practically exponentially stabilizes the closed loop despite the noise affecting the data.
After presenting the MPC scheme in Section~\ref{subsec:robust_scheme}, we provide an important technical result in Section~\ref{subsec:robust_technical}, and we prove closed-loop stability in Section~\ref{subsec:robust_stability}.
For simplicity, we do not consider output constraints in this section, i.e., $\mathbb{Y}=\mathbb{R}^p$, but we conjecture that an extension of the presented results to closed-loop output constraint satisfaction is possible following similar arguments as in~\cite{berberich2020constraints}.

\subsection{Robust data-driven MPC scheme}\label{subsec:robust_scheme}

We assume the output data used for prediction via Lemma~\ref{lem:FundamentalLemma} is given by $\tilde{y}_t^\mathrm{d} = y_t^\mathrm{d} + \varepsilon_t^\mathrm{d}$, where the noise is bounded by $\| \varepsilon_t^\mathrm{d} \|_\infty \leq \bar{\varepsilon}$ for $t \geq 0$. 
Similarly, the measured output values used for the initial conditions are perturbed as $\tilde{y}_t = y_t + \varepsilon_t$, again with the noise bounded by $\| \varepsilon_t \|_\infty \leq \bar{\varepsilon}$ for $t \geq 0$.
Given I/O data $\{ u_k^\mathrm{d},\tilde{y}_k^\mathrm{d} \}_{k=0}^{N-1}$ and past I/O measurements $(u_{[t-l,t-1]},\tilde{y}_{[t-l,t-1]})$, the robust MPC problem is defined as
\begin{subequations}
	\begin{align}
		J_{L}^{*}&\left(u_{[t-l, t-1]}, \tilde{y}_{[t-l, t-1]}\right) = \label{eq:robustDDUNCMPC_Cost}\\
		\min _{\alpha(t), \sigma(t) \atop \bar{u}(t), \bar{y}(t)} &\sum_{k=0}^{L-1} \| \bar{u}_k(t) \|_R^2 + \| \bar{y}_k(t) \|_Q^2 + \lambda_\alpha \bar{\varepsilon} \| \alpha(t) \|_2^2 + \frac{\lambda_\sigma}{\bar{\varepsilon}} \| \sigma(t) \|_2^2  \nonumber\\
		\text{s.t.} &\left[\begin{array}{c}\bar{u}(t) \\ \bar{y}(t) + \sigma(t)\end{array}\right]=\left[\begin{array}{c}H_{L+l}(u^\mathrm{d})\\ H_{L+l}(\tilde{y}^\mathrm{d})\end{array}\right]\alpha(t), \label{eq:robustDDUNCMPC_Dynamics}\\
		& \left[\begin{array}{c}\bar{u}_{[-l, -1]}(t) \\ \bar{y}_{[-l,-1]}(t)\end{array}\right] = \left[\begin{array}{c}u_{[t-l, t-1]}\\ \tilde{y}_{[t-l,t-1]}\end{array}\right], \label{eq:robustDDUNCMPC_IC}\\ 
		& \bar{u}_{k}(t) \in \mathbb{U},\ k \in \mathbb{I}_{[0, L-1]}.
	\end{align}
	\label{eq:robustDDUNCMPC}
\end{subequations}
An important difference between the nominal scheme \eqref{eq:nominalDDUNCMPC} and the robust scheme \eqref{eq:robustDDUNCMPC} is that in the latter, the data in the prediction model \eqref{eq:robustDDUNCMPC_Dynamics} and the initial condition \eqref{eq:robustDDUNCMPC_IC} are affected by noise.
Thus, the considered control problem can be interpreted as a noisy output-feedback problem with multiplicative model uncertainty.
In order to account for the noise in~\eqref{eq:robustDDUNCMPC_Dynamics}, Problem~\eqref{eq:robustDDUNCMPC} contains a slack variable $\sigma$, which is regularized in the cost with some parameter $\lambda_{\sigma}>0$.
Moreover, a regularizing cost on the squared Euclidean norm of $\alpha$ is introduced with some parameter $\lambda_{\alpha}>0$, which decreases the influence of noise on the predicted trajectories in~\eqref{eq:robustDDUNCMPC_Dynamics} (compare also~\cite{DeePC1,Stability&RobustnessTEC}).
Note that the regularization parameters depend on the noise bound $\bar{\varepsilon}$ and therefore, the robust scheme \eqref{eq:robustDDUNCMPC} reduces to the nominal one for $\bar{\varepsilon}\rightarrow0$.
Further, if the input constraint set $\mathbb{U}$ is a convex polytope, then Problem \eqref{eq:robustDDUNCMPC} is a strictly convex quadratic program, which can be solved efficiently.
We write $\bar{u}^*(t)$, $\bar{y}^*(t)$, $\alpha^*(t)$, $\sigma^*(t)$ for the optimal solution of~\eqref{eq:robustDDUNCMPC} at time $t$.

In this section, we analyze the closed-loop properties resulting from the receding horizon MPC implementation based on Problem~\eqref{eq:robustDDUNCMPC}, see Algorithm~\ref{alg:robustMPC} below.

\begin{algorithm}[H]
	\caption{Robust Data-Driven MPC Scheme}
	\label{alg:robustMPC}
	\begin{algorithmic}[1]
		\Statex \textbf{Given:} horizon length $L$, I/O data $\{u_k^\mathrm{d},\tilde{y}_k^\mathrm{d}\}_{k=0}^{N-1}$, where $u^\mathrm{d}_{[0,N-1]}$ is persistently exciting of order $L+l+n$, constraint set $\mathbb{U}$, I/O weighting matrices $Q, R$, regularization parameters $\lambda_\alpha, \lambda_\sigma$ and noise bound $\bar{\varepsilon}$.
		\State  At time $t$, take the past $l$ measurements $(u_{[t-l,t-1]}$, $\tilde{y}_{[t-l,t-1]})$ and solve~\eqref{eq:robustDDUNCMPC}.
		\State Apply the input $u_t=\bar{u}_0^*(t)$.
		\State Set $t \coloneqq t+1$ and go back to 1.
	\end{algorithmic}
\end{algorithm}

We note that the MPC scheme considered in this section is the same as in~\cite{DeePC1,DeePC4}, wherein open-loop robustness properties are shown using probabilistic arguments.
Thus, our results can also be seen as providing \emph{closed-loop} guarantees of the algorithm first proposed in~\cite{DeePC1} which has sparked increasing interest in the recent literature.
%Algorithm~\ref{alg:robustMPC} is also similar to the MPC approach considered in~\cite{Stability&RobustnessTEC} with the main difference that the latter includes stabilizing terminal equality constraints, which can deteriorate the closed-loop performance and robustness if compared to a scheme without terminal ingredients.
%Another important advantage of the presented approach over the one in~\cite{Stability&RobustnessTEC} is that we provide closed-loop stability and robustness guarantees for the \emph{one-step} MPC scheme in Algorithm~\ref{alg:robustMPC}, whereas~\cite{Stability&RobustnessTEC} only provides such guarantees for a \emph{multi-step} MPC scheme due to the terminal equality constraints.
Since Problem~\ref{eq:robustDDUNCMPC} does not contain any stabilizing terminal ingredients and due to the noisy output measurements, the stability analysis is non-trivial and divided in the derivation of a continuity-like property of the Lyapunov function (Section~\ref{subsec:robust_technical}) and the actual stability proof (Section~\ref{subsec:robust_stability}).

%Finally, we note that the MPC proposed in~\cite{Stability&RobustnessTEC} included the \emph{non-convex} constraint 
%%
%\begin{align}\label{eq:non_convex}
%\lVert\sigma(t)\rVert_{\infty}\leq\bar{\varepsilon}(1+\lVert\alpha(t)\rVert_1),
%\end{align}
%%
%which was required to prove closed-loop stability in~\cite{Stability&RobustnessTEC}.
%On the contrary, Problem~\eqref{eq:robustDDUNCMPC} as well as the theoretical closed-loop guarantees derived in the remainder of this section do not require such a non-convex constraint.
%In particular, given the chosen regularization, the slack variable $\sigma$ is automatically small if the noise level is small.
%Similarly, the stability results for data-driven MPC with terminal equality constraints in~\cite{Stability&RobustnessTEC} remain true if the non-convex constraint~\eqref{eq:non_convex} is dropped but the regularization of $\sigma(t)$ takes the form $\frac{\lambda_{\sigma}}{\bar{\varepsilon}}\lVert\sigma(t)\rVert_2^2$.

\begin{remark}
Note that the proposed MPC scheme penalizes the difference of the predicted trajectory w.r.t. zero, i.e., we only consider stabilization of the origin.
All results in this section hold qualitatively for non-zero setpoints with a lower noise level $\bar{\varepsilon}$ since the noise acts as a multiplicative uncertainty in~\eqref{eq:robustDDUNCMPC_Dynamics} (compare also~\cite[Remark 5]{Stability&RobustnessTEC} for a more detailed discussion of this issue for robust data-driven MPC with terminal equality constraints).
For the implementation of the corresponding MPC scheme, it needs to be verified whether the given input-output setpoint is indeed an equilibrium, e.g., via a prior experiment in case of open-loop stable systems.
In case an exact equilibrium point of the unknown plant is not available, one can also leverage a data-driven MPC formulation with artificial setpoints, cf.~\cite{berberich2020tracking,berberich2021linearpart2}.
\end{remark}

\subsection{Main technical result}\label{subsec:robust_technical}

We denote the perturbed extended state $\tilde{\xi}_t$ by
\begin{align}
	\tilde{\xi}_t \coloneqq 
	\begin{bmatrix}
		u_{[t-l,t-1]}\\
		\tilde{y}_{[t-l,t-1]}
	\end{bmatrix}
	= \xi_t +
	\begin{bmatrix}
		0_{lm \times 1}\\
		\varepsilon_{[t-l,t-1]}
	\end{bmatrix}.
\end{align}

The following technical result proves a continuity-like property of the Lyapunov function candidate $Y_L(\tilde{\xi}_t) \coloneqq J_L^*(\tilde{\xi}_t)+ W(\tilde{\xi}_t)$, where $W$ is an IOSS Lyapunov function (compare Lemma~\ref{lemma:extended}).
The result is inspired by previous work on \emph{model-based} robust MPC with state measurements~\cite{NovelConstraintTightening}, and it is useful in the proof of Theorem \ref{thm:practicalStabilityThm} below.

\begin{lemma}\label{lem:robustStabLemma}
	Assume Problem \eqref{eq:robustDDUNCMPC} is feasible at time $t$, $u^\mathrm{d}$ is persistently exciting of order $l+L+n$, and  $Y_L(\tilde{\xi}_{t}) \leq \bar{Y}$.
	Then, the function $Y_{L-1}(\xi)$ satisfies
	\begin{align} \label{eq:ContLikePropOfLyap}
		Y_{L-1}(\tilde{\xi}_{t+1}) \leq Y_L(\tilde{\xi}_t) - \epsilon_{\mathrm{o}} \| \tilde{\xi}_t \|_2^2 + \alpha_3(\bar{\varepsilon}),
	\end{align}
	where $\alpha_3 \in \mathcal{K}_\infty$ and with $\epsilon_0>0$ as in Lemma~\ref{lemma:extended}.
\end{lemma}
The detailed proof of Lemma~\ref{lem:robustStabLemma} can be found in Appendix~\ref{app:proof}.
In the proof, we first quantify the prediction error caused by the noise and the slack variable in~\eqref{eq:robustDDUNCMPC_Dynamics} and~\eqref{eq:robustDDUNCMPC_IC}.
Based on the resulting error bound, we construct a feasible candidate solution to upper bound the difference between the value functions $J_{L-1}^*(\tilde{\xi}_{t+1})-J_L^*(\tilde{\xi}_t)$. 
%continuity-like property 
By combining this candidate with a local continuity bound of the IOSS Lyapunov function $W$, we then obtain Inequality~\eqref{eq:ContLikePropOfLyap}.
Note that the ``error term'' $\alpha_3(\bar{\varepsilon})$ on the right-hand side of~\eqref{eq:ContLikePropOfLyap} approaches zero if the noise bound $\bar{\varepsilon}$ approaches zero. 
This fact is crucial for establishing closed-loop practical stability w.r.t. the noise level in Section~\ref{subsec:robust_stability}.

\subsection{Closed-loop stability guarantees}\label{subsec:robust_stability}
The following result provides theoretical guarantees on recursive feasibility, input constraint satisfaction and practical exponential stability for the closed loop under Algorithm \ref{alg:robustMPC}.

\begin{theorem}[Robust stability guarantees] \label{thm:practicalStabilityThm}
	\begin{subequations}
\label{eq:PSTBound_both}
	Assume 
	%Problem \eqref{eq:robustDDUNCMPC} is feasible at time $t = 0$, 
	$u^\mathrm{d}$ is persistently exciting of order $l+L+n$ and  $Y_L(\tilde{\xi}_{0}) \leq \bar{Y}$. 
	Then, there exists a constant $\tilde{L}_{\bar{Y}} \ge 0$ such that for any horizon length $L > \tilde{L}_{\bar{Y}}$, there exists a noise bound $\hat{\varepsilon}_L>0$ such that for any noise bound $\bar{\varepsilon} \leq \hat{\varepsilon}_L$, Problem \eqref{eq:robustDDUNCMPC} is feasible for all times $t \geq 0$, the closed-loop input satisfies the constraints, i.e., $u_t\in\mathbb{U}$, and the function $Y_L = J_L^* + W$ fulfills
		\begin{align}
			\epsilon_{\mathrm{o}} \| \tilde{\xi}_t \|_2^2 &\leq Y_L(\tilde{\xi}_t) \leq \gamma_{\bar{Y}} \| \tilde{\xi}_t \|_2^2 + \alpha_Y(\bar{\varepsilon}), \label{eq:PSTBoundOnYL}\\
			Y_L(\tilde{\xi}_{t+1}) - Y_L(\tilde{\xi}_t) &\leq - \tilde{\alpha}_L \| \tilde{\xi}_t \|_2^2 + \alpha_8(\bar{\varepsilon}), \label{eq:PSTBoundOnYLIncrease}
		\end{align}
	\end{subequations}
	with $\gamma_{\bar{Y}},\tilde{\alpha}_L> 0$, $\alpha_Y,\alpha_8\in \mathcal{K}_\infty$.
	%\begin{subequations}
	%	\begin{align}
	%%		\gamma_{\bar{Y}} &> 0,\quad \alpha_Y(\bar{\varepsilon}) \in \mathcal{K}\\
	%		\tilde{\alpha}_L &\coloneqq \epsilon_{\mathrm{o}} - \frac{\gamma_s (\gamma_{\bar{Y}} - \epsilon_{\mathrm{o}})}{\epsilon_{\mathrm{o}} (L-2)} > 0,
	%%	\end{align}
	%\end{subequations}
	%where for any fixed $L$, $\alpha_6(\cdot, L)$ is a class $\mathcal{K}_\infty$-function.
%	Furthermore, the set $\mathbb{B}_{\epsilon} \in \R^{l(m+p)}$ with
%	\begin{align}
%		\epsilon = \frac{1}{\epsilon_{\mathrm{o}}} \left[ \gamma_{\bar{Y}} \alpha_6(\bar{\varepsilon},L) \left( \frac{1}{\tilde{\alpha}_L} + 1 \right) + \alpha_Y(\bar{\varepsilon}) + \alpha_6(\bar{\varepsilon},L) \right]
%	\end{align}
%	is exponentially stable.
\end{theorem}
\begin{proof}
	The proof is partitioned into three parts. The first part establishes the lower and upper bounds on $Y_L$~\eqref{eq:PSTBoundOnYL}. 
	The second part uses the continuity-like properties of the Lyapunov candidate function $Y_L$ in Lemma~\ref{lem:robustStabLemma} to contruct a nominal feasible trajectory  at the next time step $t+1$.
%	show existence of a horizon length $L$, and a noise bound $\bar{\varepsilon}$ ensuring feasibility at the next time step $t+1$. 
The third and last part uses previous results to establish the practical Lyapunov inequality \eqref{eq:PSTBoundOnYLIncrease}. \\
%The ideas in the proof draw from ideas in the literature about MPC without terminal ingredients~\cite{ImplicitSolutions} and robust MPC \cite{NovelConstraintTightening}.\\
%unconstrained MPC
	\textbf{Part I}
	The lower bound in \eqref{eq:PSTBoundOnYL} is trivial to show using the IOSS property \eqref{eq:IOSS}, which holds since $W$ is an IOSS Lyapunov function for the nominal system~\eqref{eq:xiDynamics} and $\tilde{\xi}$ can be viewed as a state thereof.
	To show the upper bound in \eqref{eq:PSTBoundOnYL}, we utilize a result in a study analyzing a similar scheme to \eqref{eq:robustDDUNCMPC} with additional terminal constraints \cite{Stability&RobustnessTEC}.
In \cite[Lemma 1]{Stability&RobustnessTEC}, it is shown that that there exists a constant $\tilde{\delta}>0$, such that for any $\tilde{\xi}_t\in\mathbb{B}_{\tilde{\delta}}$ the data-driven MPC is feasible and the Lyapunov function satisfies the following bound\footnote{%
Compared to~\cite{Stability&RobustnessTEC}, the regularization terms in~\eqref{eq:robustDDUNCMPC} have an additional scaling w.r.t. $\bar{\varepsilon}$, which leads to a small modification in the bound.
}
\begin{align}
\label{eq:robust_local_upper_bound}
Y_L(\tilde{\xi}_t)\leq \gamma_{\mathrm{s}}\|\tilde{\xi}_t\|_2^2+c_2\bar{\varepsilon},
\end{align}
with constants $\gamma_{\mathrm{s}},c_2>0$. 
Since the optimization problem in~\cite{Stability&RobustnessTEC} only has an additional terminal constraint, the corresponding solution is also a feasible candidate solution to problem~\eqref{eq:robustDDUNCMPC} and thus the same arguments can be used to show that~\eqref{eq:robust_local_upper_bound} holds for any $\tilde{\xi}_t\in\mathbb{B}_{\tilde{\delta}}$. Analogous to part (ii) in the proof of~\cite[Thm.~3]{Stability&RobustnessTEC} this local upper bound ensures the upper bound~\eqref{eq:PSTBoundOnYL} for $Y_L\leq \bar{Y}$ with $\gamma_{\bar{Y}}:=\max \left\{ \gamma_{\mathrm{s}},\dfrac{\bar{Y}-c_2\bar{\varepsilon}}{\tilde{\delta}^2} \right\}$, $\alpha_Y(\bar{\varepsilon}):=c_2\bar{\varepsilon}$. \\	
	\textbf{Part II} Using the continuity-like property of the Lyapunov candidate $Y_L$ \eqref{eq:ContLikePropOfLyap} and the upper bound~\eqref{eq:PSTBoundOnYL}, the shortened Lyapunov candidate after one step $Y_{L-1}(\tilde{\xi}_{t+1})$ satisfies
	\begin{align} \label{eq:PSTContLikeOfLyap1}
		\begin{split}
			Y_{L-1}(\tilde{\xi}_{t+1}) &\leq \underbrace{Y_L(\tilde{\xi}_t)}_{\leq \gamma_{\bar{Y}} \| \tilde{\xi}_t \|_2^2 + \alpha_Y(\bar{\varepsilon})} - \epsilon_{\mathrm{o}} \| \tilde{\xi}_t \|_2^2 + \alpha_3(\bar{\varepsilon})\\
			&\leq (\gamma_{\bar{Y}} - \epsilon_{\mathrm{o}}) \| \tilde{\xi}_t \|_2^2 + \underbrace{\alpha_Y(\bar{\varepsilon}) + \alpha_3(\bar{\varepsilon})}_{\eqqcolon \alpha_4(\bar{\varepsilon})}.
		\end{split}
	\end{align}
	Note that the lower and upper bounds on $Y_L$ \eqref{eq:PSTBoundOnYL} directly imply that $\gamma_{\bar{Y}} - \epsilon_{\mathrm{o}} \geq 0$. 
	
	The next step consists of showing that for a long enough prediction horizon $L$, the shortened problem after one step is still feasible given that the original problem is feasible. To this end, a bound on the predicted states over the shortened horizon $L-1$ needs to be given. Therefore, we use conceptually similar steps as in the proof of \cite[Theorem 1]{ImplicitSolutions}.
	
	Consider the predicted state trajectory $\{\bar{\xi}_{k}(t+1)\}_{k=0}^{L-1}$ corresponding to the nominal input and output candidate trajectory \eqref{eq:NewCandInput}, \eqref{eq:NewOutputCandidate} used in the proof of Lemma~\ref{lem:robustStabLemma}, i.e., $\bar{\xi}_k(t+1)=\begin{bmatrix}\bar{u}_{[k-l,k-1]}(t+1)\\\bar{y}_{[k-l,k-1]}(t+1)	\end{bmatrix}$, $k\in\mathbb{I}_{[0,L-1]}$.
	For this nominal trajectory, the detectability property \eqref{eq:IOSS} yields
	\begin{align}
			&W(\bar{\xi}_{L-1}(t+1)) - W(\tilde{\xi}_{t+1})\\
			=&\sum_{k = 0}^{L-2} W(\bar{\xi}_{k+1}(t+1)) - W(\bar{\xi}_k(t+1))\nonumber\\
			\leq& - \epsilon_{\mathrm{o}}\sum_{k = 0}^{L-2} \| \bar{\xi}_k(t+1) \|_2^2+\underbrace{\sum_{k = 0}^{L-2} \ell(\bar{u}_k^*(t+1),\bar{y}_k^*(t+1))}_{\leq J_{L-1}^*(\tilde{\xi}_{t+1})}+\alpha_5(\bar{\varepsilon})\nonumber
	\end{align}%
	with some $\alpha_5\in\mathcal{K}_{\infty}$, where the last inequality can be shown similarly to the proof of Lemma~\ref{lem:robustStabLemma} using the fact that $\lVert \bar{y}_k(t+1)-\bar{y}_k^*(t+1)\rVert_Q$ can be bounded using class $\mathcal{K}_\infty$ functions w.r.t. $\bar{\varepsilon}$ (compare~\eqref{eq:DeltaYSquaredBound}).
%	where the last inequality can be shown similarly to the proof of Lemma~\ref{lem:robustStabLemma} using that $\lVert\xi_t-\tilde{\xi}_t\rVert_2^2$ and $\lVert \bar{y}_k(t+1)-\bar{y}_k^*(t+1)\rVert_Q^2$ are both bounded by $\bar{\alpha}_4(\bar{\varepsilon})$ for some $\bar{\alpha}_4\in\mathcal{K}_{\infty}$ (compare~\eqref{eq:DeltaYSquaredBound}).
	Since $W(\xi) \geq 0$, we obtain
	\begin{align}
		\epsilon_{\mathrm{o}} \sum_{k = 0}^{L-2} \| \bar{\xi}_k(t+1) \|_2^2 &\leq \underbrace{W(\tilde{\xi}_{t+1}) + J_{L-1}^*(\tilde{\xi}_{t+1})}_{= Y_{L-1}(\tilde{\xi}_{t+1})}+\alpha_5(\bar{\varepsilon}).
	\end{align}
Combining this bound with Inequality~\eqref{eq:PSTContLikeOfLyap1} yields 
	\begin{align} \label{eq:PSTSumOfPredStates}
		\epsilon_{\mathrm{o}} \sum_{k = 0}^{L-2} \| \bar{\xi}_k(t+1) \|_2^2 \leq (\gamma_{\bar{Y}} - \epsilon_{\mathrm{o}}) \| \tilde{\xi}_t \|_2^2 + \alpha_6(\bar{\varepsilon})
	\end{align}
	for $\alpha_6(\bar{\varepsilon})\coloneqq\alpha_4(\bar{\varepsilon})+\alpha_5(\bar{\varepsilon})$.
	%The bound on the sum in \eqref{eq:PSTSumOfPredStates} implies that there exists one summand which must be lower or equal to the bound divided by the number of summands
Given the sum with non-negative summands, a simple proof of contradiction shows that at least one of the summands has to be smaller than or equal to the average of the sum.
%	To show the boundedness of measure of at least one predicted state, we use that for a bounded sum there must exist at least one summand lower than or equal to the average of the sum. This can be shown by a simple proof of the contrapositive.
	Thus, there exists at least one $k_x \in \mathbb{I}_{[0, L-2]}$ s.\,t.
	\begin{align} \label{eq:xiAppendLQRBound}
		\| \bar{\xi}_{k_x}(t+1) \|_2^2 &\leq \frac{(\gamma_{\bar{Y}} - \epsilon_{\mathrm{o}}) \| \tilde{\xi}_t \|_2^2 + \alpha_6(\bar{\varepsilon})}{\epsilon_{\mathrm{o}} (L-1)}.%\nonumber\\
%		&\stackrel{\eqref{eq:PSTBoundOnYL}}{\leq} \frac{(\gamma_{\bar{Y}} - \epsilon_{\mathrm{o}}) \frac{\bar{Y}}{\epsilon_{\mathrm{o}}} + w}{\epsilon_{\mathrm{o}} (L-1)},
	\end{align}
%where $w>0$ is a constant fixed later in the proof, and using  $\alpha_6(\bar{\varepsilon}) \leq w$ for $\bar{\varepsilon}\leq\alpha_6^{-1}(w)$.
Using $Y_L\leq \bar{Y}$ instead of Inequality~\eqref{eq:PSTContLikeOfLyap1}, we additionally have
	\begin{align} \label{eq:xiAppendLQRBound_bar}
		\| \bar{\xi}_{k_x}(t+1) \|_2^2 &\leq\frac{\bar{Y}+w}{\epsilon_{\mathrm{o}} (L-1)}.
	\end{align}
Here, $w>0$ is an arbitrary but fixed constant satisfying $w\geq\alpha_3(\bar{\varepsilon})+\alpha_5(\bar{\varepsilon})$.
Later in the proof, we require $\bar{\varepsilon}$ to be sufficiently small such that $w$ can be chosen arbitrarily small as well.
In particular,
%, and using  $\alpha_6(\bar{\varepsilon}) \leq w$ for $\bar{\varepsilon}\leq\alpha_6^{-1}(w)$.
%	using $\| \tilde{\xi}_t \|_2^2 \leq \frac{Y_L(\tilde{\xi}_t)}{\epsilon_{\mathrm{o}}} \leq \frac{\bar{Y}}{\epsilon_{\mathrm{o}}}$ and the noise bound $\alpha_4(\bar{\varepsilon}, L) \leq w$, where $w$ is an arbitrary positive constant.
%	Note that for any fixed, finite horizon $\hat{L}$, the function $\hat{\alpha}(\bar{\varepsilon}) \coloneqq \alpha_4(\bar{\varepsilon}, \hat{L})$ is a class $\mathcal{K}_\infty$-function, and hence, there always exists a noise bound $\bar{\varepsilon} \leq \hat{\varepsilon}_L$ with $\hat{\varepsilon}_L \coloneqq \hat{\alpha}^{-1}(w)$, such that $\alpha_4(\hat{\varepsilon}_L, \hat{L}) \leq w$.\\
%	This means that for all horizons $L$ which fulfill $L > \tilde{L}_0 \coloneqq 2 + \frac{(\gamma_{\bar{Y}} - \epsilon_{\mathrm{o}}) \frac{\bar{Y}}{\epsilon_{\mathrm{o}}} + w}{\epsilon_{\mathrm{o}}} \frac{1}{\delta^2}$, it holds that $\| \bar{\xi}_{k_x}(t+1) \|_2 \leq \delta$ and therefore $\bar{\xi}_{k_x}(t+1) \in \mathbb{B}_\delta$. 
	%Note that 
	for $L > \tilde{L}_0 \coloneqq 1 + \dfrac{\bar{Y}}{\epsilon_{\mathrm{o}}\delta^2}$, there exists a small enough constant $w>0$ such that~\eqref{eq:xiAppendLQRBound} ensures $\bar{\xi}_{k_x}(t+1) \in \mathbb{B}_\delta$ with $\delta$ from the proof of Theorem~\ref{thm:NominalStabilityGuarantees}.
It can be shown that starting at $\bar{\xi}_{k_x}(t+1)$, appending the input $\bar{u}_k(t+1)=K\bar{\xi}_{k}(t+1)$, $k\geq k_x$, results in a nominally feasible trajectory for all future times, similar to the proof of Theorem \ref{thm:NominalStabilityGuarantees}.\\
%	Hence, using similar arguments to Lemma \ref{lem:localStabilizabilityProof}, it can be shown that starting at $\bar{\xi}_{k_x}(t+1)$ the input by the LQR is feasible for all future times, similar to the proof of Theorem \ref{thm:NominalStabilityGuarantees}.
	\textbf{Part III} 
In the following, we derive a bound on the value function $J_L^*(\tilde{\xi}_{t+1})$  using the candidate input trajectory defined in the previous part. 
%	The basic idea for the following steps is to bound the new full-length value function $J_L^*(\tilde{\xi}_{t+1})$ using the previous bounds as well as local stabilizability.
%	subject to the same set of constraints as the cost \eqref{eq:robustDDUNCMPC_Cost}.
The candidate input trajectory at time $t+1$  consists of the initial part of the previously optimal input trajectory shifted one and appended by the linear controller $u^{\mathrm{loc}}=K\xi$, i.e.,
%	Consider a new candidate solution as follows. At time $t+1$, the new candidate trajectory consists of the previously optimal input trajectory shifted 
%	\begin{align}
%			\bar{u}_{[-l,k_x-1]}(t+1) = \bar{u}_{[-l+1,k_x]}^*(t),
%	\end{align}
%	and a feasible input trajectory generated by a \JK{linear} controller $\bar{u}^\text{loc}(t+1)$ starting at step $k_x$ as in the last part:
	\begin{align} \label{eq:NewFullLengthInput}
		\bar{u}_{[-l,L-1]}(t+1) = 
		\begin{bmatrix}
			\bar{u}_{[-l+1,k_x]}^*(t)\\
			\bar{u}^{\mathrm{loc}}_{[k_x, L-1]}(t+1)
		\end{bmatrix}.
	\end{align}
	The following steps construct a corresponding feasible candidate solution for $\bar{y},\alpha,\sigma$ for Problem~\eqref{eq:robustDDUNCMPC} analogous to the proof of Lemma~\ref{lem:robustStabLemma}.
	First, the vector $\alpha$ corresponding to the chosen input candidate is constructed.
	Consider the Hankel matrix
	\begin{align}
		H_{ux} \coloneqq \begin{bmatrix}
			H_{l+L}(u^\mathrm{d})\\
			H_1(x^\mathrm{d}_{[0,N-L-l]})
		\end{bmatrix},
	\end{align}
	which has full row rank since $u^\mathrm{d}$ is persistently exciting of order $n+L$ according to \cite[Corollary 2, (iii)]{Willems05}. The sequence $x^\mathrm{d}$ is uniquely determined by the sequences $u^\mathrm{d}, y^\mathrm{d}$ since system \eqref{eq:MinimalLTIsystem} is observable and $l\geq\underline{l}$. Since $H_{ux}$ has full row rank, there exists the right-inverse
	\begin{align}
		H_{ux}^\dagger = H_{ux}^\top \left(H_{ux} H_{ux}^\top \right)^{-1}.
	\end{align}
%	\begin{align}
%		\begin{bmatrix}
%			\bar{u}_{[-l,L-1]}(t+1)\\
%			x_{t+1-l}
%		\end{bmatrix}
%		= H_{ux} \bar{\alpha}(t+1),
%	\end{align}
%	where $\bar{\alpha}$ represents the new candidate vector $\alpha$ from Lemma \ref{lem:FundamentalLemma}.
We define $\bar{\alpha}(t+1)$ based on the right-inverse $H_{ux}^\dagger$ as
	\begin{align} \label{eq:NewFullLengthAlpha}
		\bar{\alpha}(t+1) = H_{ux}^\dagger
		\begin{bmatrix}
			\bar{u}_{[-l,L-1]}(t+1)\\
			x_{t+1-l}
		\end{bmatrix}.
	\end{align}
	The output candidate trajectory is chosen as the corresponding nominal output trajectory with the initial condition such that \eqref{eq:robustDDUNCMPC_IC} is satisfied, i.e.,
	\begin{align} \label{eq:NewFullLengthOutput}
		&\bar{y}_{[-l,-1]}(t+1) = \tilde{y}_{[t-l,t-1]}\\
		&\bar{y}_{[0,L-1]}(t+1) = H_{L}(y^\mathrm{d}) \bar{\alpha}(t+1). \nonumber
	\end{align}
	The slack variable $\sigma$ is chosen as
	\begin{align} \label{eq:NewFullLengthSlackCandidate}
		\bar{\sigma}_{[-l,L-1]}(t+1) = H_{L+l}(\varepsilon^\mathrm{d}) \bar{\alpha}(t+1) - \begin{bmatrix}
			\varepsilon_{[-l,-1]}\\
			0_{pL \times 1}
		\end{bmatrix},
	\end{align}
	satisfying the constraints \eqref{eq:robustDDUNCMPC_Dynamics} and \eqref{eq:robustDDUNCMPC_IC}.
%	The considered two-part partition of the new candidate trajectory as in \eqref{eq:NewFullLengthInput} allows to express its cost as follows.
	Define the I/O cost $\hat{J}_L(\bar{u},\bar{y})$ related to the cost from~\eqref{eq:robustDDUNCMPC} without the regularization terms as
	\begin{align}
		\hat{J}_L(\bar{u}_{[0,L-1]}(t), \bar{y}_{[0,L-1]}(t)) \coloneqq \sum_{k=0}^{L-1} \| \bar{u}_k(t) \|_R^2 + \| \bar{y}_k(t) \|_Q^2.
	\end{align}
Using exponential stability of the appended linear control input $u^{\mathrm{loc}}$, the appended cost satisfies
	\begin{align}
	\label{eq:cost_appended}
&		\hat{J}_{L-k_x}(\bar{u}_{[k_x,L-1]}(t), \bar{y}_{[k_x,L-1]}(t)) \leq \gamma_{\mathrm{s}} \| \bar{\xi}_{k_x}(t+1) \|_2^2\nonumber\\
		\overset{\eqref{eq:xiAppendLQRBound}}{\leq}& \frac{\gamma_{\mathrm{s}}(\gamma_{\bar{Y}}-\epsilon_{\mathrm{o}})}{\epsilon_{\mathrm{o}}(L-1)} \| \tilde{\xi}_t \|_2^2+ \frac{\gamma_{\mathrm{s}}}{\epsilon_{\mathrm{o}}(L-1)} \alpha_6(\bar{\varepsilon}),
	\end{align}
with $\gamma_{\mathrm{s}}>0$ as in~\eqref{eq:thm_nominal_gammas}. 	
Next, we bound the regularization terms in the cost \eqref{eq:robustDDUNCMPC_Cost} based on similar arguments as in the proof of Lemma~\ref{lem:robustStabLemma}, where we considered a similar candidate solution for the optimal control problem~\eqref{eq:robustDDUNCMPC} with horizon $L-1$. 
The following norm bound on the candidate $\bar{\alpha}(t+1)$ in \eqref{eq:NewFullLengthAlpha} is straightforward to derive:
	\begin{align}
		\| \bar{\alpha}(t+1) \|_2^2 \leq \| H_{ux}^\dagger \|_2^2 \left( \| \bar{u}_{[-l,L-1]}(t+1) \|_2^2 + \| x_{t+1-l} \|_2^2 \right).
	\end{align}
	The norm of the two-part input candidate \eqref{eq:NewFullLengthInput} is bounded by
	\begin{align} \label{eq:FullLengthInputBound}
		\| \bar{u}_{[-l,L-1]}(t+1) \|_2^2 \leq \frac{\bar{Y}}{\lambda_{\min}(R)} + \frac{\gamma_{\mathrm{s}} \delta^2}{\lambda_{\min}(R)},
	\end{align}
	using $Y_L(\tilde{\xi}_t) \leq \bar{Y}$, \eqref{eq:cost_appended}, and $\bar{\xi}_{k_x}(t+1) \in \mathbb{B}_\delta$.
%	\JK{\\@JB: das bound gilt so, erste $L$ schritte input sind kleiner $Y$, appended ist kleiner $\gamma$ $\|\xi_{k_x}\|^2$ und in Part I wurd $\|\xi_{k_x}\|\leq \delta$ gezeigt, noise taucht nicht auf das es nur eine input trajectorie ist\\}
	% $\bar{\xi}_{k_x}(t+1) \in \mathbb{B}_\delta$ with Lemma~\ref{lem:localStabilizabilityProof}.
	Following the same steps as in the proof of Lemma~\ref{lem:robustStabLemma} (see~\eqref{eq:app_x_t_l_2}), %it can be derived that
	we obtain the uniform bound
	\begin{align} \label{eq:PastStateBound}
		\| x_{t+1-l} \|_2 \leq c_x
	\end{align}
	for some $c_x > 0$.
	Therefore, the norm of the candidate $\bar{\alpha}(t+1)$ is bounded by
	\begin{align} \label{eq:NewFullLengthAlphaBound}
		\| \bar{\alpha}(t+1) \|_2^2 \overset{\eqref{eq:NewFullLengthAlpha}}{\leq}& \| H_{ux}^\dagger \|_2^2 \left( \| \bar{u}_{[-l,L-1]}(t+1) \|_2^2 + \| x_{t+1-l} \|_2^2 \right) \nonumber\\
		\overset{\eqref{eq:FullLengthInputBound},\eqref{eq:PastStateBound}}{\leq}& \underbrace{\| H_{ux}^\dagger \|_2^2 \left( \frac{\bar{Y}}{\lambda_{\min}(R)} + \frac{\gamma_{\mathrm{s}} \delta}{\lambda_{\min}(R)} + c_x^2 \right)}_{\eqqcolon c_\alpha^2}.
	\end{align}
	It can be shown exactly as in the proof of Lemma~\ref{lem:robustStabLemma} (i.e., in~\eqref{eq:bound_sigma_cand}) that the norm of the new candidate slack variable $\bar{\sigma}(t+1)$ in \eqref{eq:NewFullLengthSlackCandidate} is bounded by
	\begin{align} \label{eq:NewFullLengthSigmaBound}
		\| \bar{\sigma}(t+1) \|_2 \leq \bar{\varepsilon} b_\sigma,
	\end{align}
	with $b_\sigma>0$.
	Using the output prediction error $\delta_{k}^y(t+1) \coloneqq \bar{y}_k(t+1) - \bar{y}_{k+1}^*(t)$, 
	%as in the proof of Lemma~\ref{lem:robustStabLemma}, 
	the I/O cost of the first $k_x$ steps is bounded by
	\begin{align}
		&\hat{J}_{k_x}(\bar{u}_{[0,k_x-1]}(t+1), \bar{y}_{[0,k_x-1]}(t+1))\\\nonumber
		&= \sum_{k=0}^{k_x-1} \| \bar{u}_k(t+1) \|_R^2 + \| \bar{y}_k(t+1) \|_Q^2\\\nonumber
		&\leq \underbrace{\sum_{k=0}^{k_x-1} \| \bar{u}_{k+1}^*(t) \|_R^2 + \| \bar{y}_{k+1}^*(t) \|_Q^2}_{\leq J_L^*(\tilde{\xi}_t) - \ell(\bar{u}_0^*(t), \bar{y}_0^*(t))}\\\nonumber
		&+ \sum_{k = 0}^{k_x-1} \| \delta_{k}^y(t+1) \|_Q^2+ 2 \sum_{k = 0}^{k_x-1} \| \bar{y}_{k+1}^*(t) \|_Q \| \delta_{k}^y(t+1) \|_Q.
	\end{align}
	Utilizing the norm bounds on the output prediction error $\delta_{k}^y(t+1)$ in the proof of Lemma~\ref{lem:robustStabLemma} (see~\eqref{eq:DeltaYSquaredBound}), the I/O cost of the first $k_x$ steps is bounded by
	\begin{align}
	\begin{split}
		&\hat{J}_{k_x}(\bar{u}_{[0,k_x-1]}(t+1), \bar{y}_{[0,k_x-1]}(t+1))\\
		\leq &J_L^*(\tilde{\xi}_t) - \ell(\bar{u}_0^*(t), \bar{y}_0^*(t))\\
		&+ \underbrace{\bar{\varepsilon} d_2 + b_\xi(\bar{\varepsilon})^2 d_3 + 2 \sqrt{\bar{Y}} \sqrt{\bar{\varepsilon}} (L-1) d_1}_{\eqqcolon \alpha_7(\bar{\varepsilon})},
	\end{split}	
	\end{align}
	using $\| \bar{y}_{k+1}^*(t) \|_Q \leq \sqrt{\bar{Y}}$, and $k_x \leq L-1$, where $d_1, d_2, d_3 > 0$, and $b_\xi \in \mathcal{K}_\infty$ are defined in Appendix~\ref{app:proof} (see~\eqref{eq:b_xi}, \eqref{eq:OutputPredErrorBound1}, and \eqref{eq:DeltaYSquaredBound}). 
	The I/O cost of the candidate solution \eqref{eq:NewFullLengthInput}, \eqref{eq:NewFullLengthOutput} over the complete horizon $L$ can thus be shown to satisfy
	\begin{align}
		&\hat{J}_L(\bar{u}_{[0,L-1]}(t), \bar{y}_{[0,L-1]}(t))\nonumber\\
		\leq &J_L^*(\tilde{\xi}_t) - \ell(\bar{u}_0^*(t), \bar{y}_0^*(t)) + \frac{\gamma_{\mathrm{s}}(\gamma_{\bar{Y}}-\epsilon_{\mathrm{o}})}{\epsilon_{\mathrm{o}}(L-1)} \| \tilde{\xi}_t \|_2^2\nonumber\\
		&+ \frac{\gamma_{\mathrm{s}}}{\epsilon_{\mathrm{o}}(L-1)} \alpha_6(\bar{\varepsilon}) + \alpha_7(\bar{\varepsilon}).
	\end{align}
	Adding the cost imposed by the regularization candidates and using the bounds \eqref{eq:NewFullLengthAlphaBound}, \eqref{eq:NewFullLengthSigmaBound}, the value function satisfies %$J_L^*(\tilde{\xi}_{t+1})$
	\begin{align}\nonumber
		J_L^*(\tilde{\xi}_{t+1}){\leq} &J_L^*(\tilde{\xi}_t) - \ell(\bar{u}_0^*(t), \bar{y}_0^*(t)) + \frac{\gamma_{\mathrm{s}}(\gamma_{\bar{Y}}-\epsilon_{\mathrm{o}})}{\epsilon_{\mathrm{o}}(L-1)} \| \tilde{\xi}_t \|_2^2\\
		& + \frac{\gamma_{\mathrm{s}}}{\epsilon_{\mathrm{o}}(L-1)} \alpha_6(\bar{\varepsilon}) + \alpha_7(\bar{\varepsilon})
		 + \lambda_\alpha \bar{\varepsilon} c_\alpha^2 + \lambda_\sigma \bar{\varepsilon} b_\sigma^2.
	\end{align}
	Adding the storage function $W(\tilde{\xi}_{t+1})$ on both sides and using the respective continuity-like property of the IOSS Lyapunov function $W$ as in the proof of Lemma~\ref{lem:robustStabLemma} (see Part IV) leads to
	\begin{align} \label{eq:PSTLyapBound1}
		\begin{split}
			Y_L(\tilde{\xi}_{t+1}) \leq &\frac{\gamma_{\mathrm{s}} (\gamma_{\bar{Y}} - \epsilon_{\mathrm{o}})}{\epsilon_{\mathrm{o}}(L-1)} \| \tilde{\xi}_t \|_2^2 + J_L^*(\tilde{\xi}_t)\\
			& - \ell(\bar{u}_0^*(t),\bar{y}_0^*(t)) + W(\bar{\xi}_1^*(t)) + \alpha_8(\bar{\varepsilon})
		\end{split}
	\end{align}
	with
	\begin{align}
		\alpha_8(\bar{\varepsilon}) \coloneqq& \frac{\gamma_{\mathrm{s}}}{\epsilon_{\mathrm{o}} (L-1)} \alpha_6(\bar{\varepsilon}) + \alpha_2(\bar{\varepsilon}) + \alpha_7(\bar{\varepsilon})
		+ \lambda_\alpha \bar{\varepsilon} c_\alpha^2 + \lambda_\sigma \bar{\varepsilon} b_\sigma^2 
	\end{align}
	and $\alpha_2\in\mathcal{K}_\infty$ according to~\eqref{eq:IOSS_continuity_alpha} in Appendix~\ref{app:proof}.
	Using the IOSS property \eqref{eq:IOSS}, we have
	\begin{align} \label{eq:PSTDetectBound1}
		W(\bar{\xi}_1^*(t)) \leq W(\tilde{\xi}_t) - \epsilon_{\mathrm{o}} \| \tilde{\xi}_t \|_2^2 + \ell(\bar{u}_0^*(t),\bar{y}_0^*(t)).
	\end{align}
	Together with the bound on the Lyapunov candidate  \eqref{eq:PSTLyapBound1}, this leads to the practical Lyapunov inequality \eqref{eq:PSTBoundOnYLIncrease} with
	\begin{align}
		\tilde{\alpha}_L \coloneqq \epsilon_{\mathrm{o}} - \frac{\gamma_{\mathrm{s}} (\gamma_{\bar{Y}} - \epsilon_{\mathrm{o}})}{\epsilon_{\mathrm{o}} (L-1)}.
	\end{align}
	To ensure that $Y_L$ is a practical Lyapunov function, we need $\tilde{\alpha}_L>0$, which holds for a sufficiently long horizon $L$, i.e.,
	\begin{align}
        \label{eq:horizon_bound_stability}
		L > \tilde{L}_1 \coloneqq 1 + \frac{\gamma_{\mathrm{s}}(\gamma_{\bar{Y}}-\epsilon_{\mathrm{o}})}{\epsilon_{\mathrm{o}}^2}.
	\end{align}
Moreover, to show that the arguments in this proof hold recursively, we need that $Y_L\leq \bar{Y}$ holds recursively. 
Based on~\eqref{eq:PSTBound_both}, this can be ensured if
%	\begin{align}
$\alpha_8(\bar{\varepsilon}) \leq \frac{\bar{Y}}{\epsilon_{\mathrm{o}}}\tilde{\alpha}_L$.
%	\end{align}
In summary, the horizon length $L$ must be such that
	\begin{align}
		L > \max \left\{ \tilde{L}_0, \tilde{L}_1 \right\} \eqqcolon \tilde{L}_{\bar{Y}}
	\end{align}
and the noise bound needs to satisfy $\bar{\varepsilon}\leq\hat{\varepsilon}_L\coloneqq \min\{ \alpha_8^{-1}(\frac{\bar{Y}}{\epsilon_{\mathrm{o}}}\tilde{\alpha}_L),(\alpha_3+\alpha_5)^{-1}(w)\}$. %\alpha_6^{-1}(w)
%	In order to show practical exponential stability, a bound on the noise level is needed. 
	%	Since the one-step prediction error $\delta_1^\xi(t)$ can be bounded w.\,r.\,t. the noise level $\bar{\varepsilon}$, the function $\alpha_4(\bar{\varepsilon}, L)$ is also bounded. Since $\alpha_4(\bar{\varepsilon}, L) \leq w \ \forall t \geq t_0$ with
%	To this end, note that for
%	\begin{align}
%		\alpha_6(\bar{\varepsilon}, L) \leq \frac{\bar{Y}}{\epsilon_{\mathrm{o}}}\tilde{\alpha}_L
%	\end{align}
%	the arguments in this proof hold recursively.
	%	Denote the upper bound on the prediction error as follows
	%	\begin{align}
	%		\bar{\delta} \geq \delta_1^\xi(t)) \ \forall t \geq 0.
	%	\end{align}
%	Thus, the function $Y_L$ is a practical Lyapunov function by [Proposition 4.3]\cite{FaulwasserGrueneMueller2018}, and the closed loop is practically exponentially stable. Specifically, the set
%	\begin{align}
%		\mathcal{B}_\epsilon = \left\{ \xi \in \R^{\tilde{n}} \ | \ Y_L(\xi) \leq \epsilon \right\},
%	\end{align}
%	with
%	\begin{align}
%		\epsilon = \frac{1}{\epsilon_{\mathrm{o}}} \left[ \gamma_{\bar{Y}} \alpha_6(\bar{\varepsilon},L) \left( \frac{1}{\tilde{\alpha}_L} + 1 \right) + \alpha_Y(\bar{\varepsilon}) + \alpha_6(\bar{\varepsilon},L) \right]
%	\end{align}
%	is exponentially stable.\\
\end{proof}

Theorem~\ref{thm:practicalStabilityThm} proves that the robust data-driven MPC scheme based on Problem~\eqref{eq:robustDDUNCMPC} practically exponentially stabilizes the closed loop despite noisy output measurements.
To be precise, the inequalities \eqref{eq:PSTBoundOnYL} and \eqref{eq:PSTBoundOnYLIncrease} guarantee that $Y_L$ is a practical Lyapunov function (compare~\cite{faulwasser2018economic}), i.e., the closed-loop state trajectory converges to a region around $x=0$ whose size increases with the noise level.
In addition to persistently exciting data, closed-loop stability only requires that the prediction horizon $L$ is chosen sufficiently large and the noise bound $\bar{\varepsilon}$ is sufficiently small.
The horizon $L$ and the noise bound $\bar{\varepsilon}$ leading to closed-loop stability depend on $\bar{Y}$, i.e., on the guaranteed region of attraction of the closed loop.
This means that the closed-loop properties generally improve (i.e., the region of attraction increases and the asymptotic tracking error decreases) if $L$ is chosen larger and $\bar{\varepsilon}$ is smaller.

The proof of Theorem~\ref{thm:practicalStabilityThm} uses that the previously optimal input sequence at time $t$, resumed at time $t+1$, results in a shortened candidate trajectory of length $L-1$, which is still feasible by Lemma \ref{lem:robustStabLemma}. 
Using the continuity-like property of the value function $J_{L-1}^*$ and the IOSS Lyapunov function $W$ as described in the proof of Lemma \ref{lem:robustStabLemma}, the Lyapunov candidate function $Y_{L-1}$ at time $t+1$ cannot deviate arbitrarily from the previous value of the full-length Lyapunov candidate $Y_L$.
This is then used to show that for long enough prediction horizons and low enough noise levels, a local controller $u^{\mathrm{loc}}$ is feasible at some future predicted state $\bar{\xi}_{k_x}(t+1)$. 
Thus, a feasible new full-length candidate trajectory can be constructed by appending this local controller to the previously optimal candidate.

\begin{remark}\label{rk:relation_to_literature}
	Theorem~\ref{thm:practicalStabilityThm} can be interpreted as a first result on closed-loop stability and robustness properties of the data-driven MPC approach proposed in~\cite{DeePC1}, for which the existing literature only contains open-loop robustness results, see, e.g.,~\cite{DeePC4}.
	Results analogous to Theorem~\ref{thm:practicalStabilityThm} are provided in~\cite{Stability&RobustnessTEC} for an MPC scheme with additional stabilizing terminal ingredients or in~\cite{berberich2021linearpart2} for a tracking MPC formulation with online optimization of an artificial setpoint.
\end{remark}

\begin{remark}\label{rk:comparison_to_TEC}
The proposed data-driven MPC scheme and its theoretical analysis have multiple advantages over the existing approach from~\cite{Stability&RobustnessTEC}.
First, it is well-known that terminal equality constraints as used in~\cite{Stability&RobustnessTEC} can lead to poor robustness properties if compared to a scheme without terminal ingredients~\cite{rawlings2017model}.
Indeed, this will be illustrated with a numerical example in Section~\ref{sec:num}.
Another important advantage of the presented approach over the one in~\cite{Stability&RobustnessTEC} is that we provide closed-loop stability and robustness guarantees for the \emph{one-step} MPC scheme in Algorithm~\ref{alg:robustMPC}, whereas~\cite{Stability&RobustnessTEC} only provides such guarantees for a \emph{multi-step} MPC scheme due to the terminal equality constraints. 
Furthermore, in the important special case of open-loop stable systems, we can provide a global region of attraction (cf. Remark~\ref{rk:RoA}), which is typically not possible with a terminal equality constraint.
The only price we have to pay for these advantages is a sufficiently long prediction horizon, cf. Theorem~\ref{thm:practicalStabilityThm}, which may potentially increase the data requirements as well as the computational complexity.
Finally, we note that the MPC proposed in~\cite{Stability&RobustnessTEC} included the \emph{non-convex} constraint 
\begin{align}\label{eq:non_convex}
\lVert\sigma(t)\rVert_{\infty}\leq\bar{\varepsilon}(1+\lVert\alpha(t)\rVert_1),
\end{align}
which was required to prove closed-loop stability in~\cite{Stability&RobustnessTEC}.
On the contrary, Problem~\eqref{eq:robustDDUNCMPC} as well as the theoretical closed-loop guarantees derived in the remainder of this section do not require such a non-convex constraint.
In particular, given the chosen regularization, the slack variable $\sigma$ is automatically small if the noise level is small.
However, it should be pointed out that the stability results for data-driven MPC with terminal equality constraints in~\cite{Stability&RobustnessTEC} remain true if the non-convex constraint~\eqref{eq:non_convex} is dropped but the regularization of $\sigma(t)$ takes the form $\frac{\lambda_{\sigma}}{\bar{\varepsilon}}\lVert\sigma(t)\rVert_2^2$.
\end{remark}

\begin{remark}
	While we omit output constraints in~\eqref{eq:robustDDUNCMPC} for simplicity, these may be accounted for using soft constraints, i.e., augmenting the cost function with a quadratic penalty function which imposes high costs on predicted outputs outside their (polytopic) constraint sets~\cite{zeilinger2014soft}. %\cite{kerrigan2000soft}. %barrier  (not sure if applicable, since barrier=infinity; and kerrigan linear
	In fact, the qualitative theoretical results in Theorem~\ref{thm:practicalStabilityThm} remain true if such soft constraints are added, thus providing a data-driven MPC approach which is simple to implement, admits rigorous closed-loop stability guarantees, and incentivizes constraint satisfaction. 
Alternatively, robust satisfaction of output constraints can be ensured by using an additional constraint tightening. 
 Such a constraint tightening method has been proposed in~\cite{berberich2020constraints} for a data-driven MPC with noisy data and we conjecture that an analogous constraint tightening can also be constructed for the MPC approach in the present paper.
%	Alternatively,~\cite{berberich2020constraints} develops a constraint tightening scheme to guarantee closed-loop output constraint satisfaction for data-driven MPC with noisy data and terminal equality constraints.
%	We conjecture that an analogous constraint tightening with similar properties can also be constructed for the MPC approach in the present paper.
\end{remark}

\begin{remark}
The constants in~\eqref{eq:PSTBoundOnYL} and~\eqref{eq:PSTBoundOnYLIncrease} which determine the guaranteed closed-loop performance are analogous to the corresponding values for the nominal MPC scheme in Theorem~\ref{thm:NominalStabilityGuarantees}, i.e., those appearing in~\eqref{eq:NominalLyapunov}.
Bounds on the constants in Theorem~\ref{thm:NominalStabilityGuarantees} can be computed from noisy data by using robust dissipativity analysis and controller design methods from~\cite{ProvablyRobustVerificationOfDissipativity,berberich2020combining}.
This means that the performance constants in Theorem~\ref{thm:practicalStabilityThm} and hence, bounds on a sufficiently long prediction horizon $L_{\bar{Y}}$ leading to closed-loop stability can also be computed based only on measured data affected by noise.
However, the resulting horizon bounds can be very conservative and the upper bound $\hat{\varepsilon}_L$ on the noise level ensuring closed-loop stability can generally not be computed without additional model knowledge, similar to the results on data-driven MPC with terminal equality constraints by~\cite{Stability&RobustnessTEC}.
Nevertheless, our results provide additional insights into the influence of system and design parameters on the closed-loop behavior.
Further exploring quantitative guidelines for tuning the involved parameters, in particular the prediction horizon $L$, remains an interesting issue for future research.
\end{remark}

\begin{remark}
As an alternative to our direct approach, the measured data can also be used to first estimate an input-output model of the underlying system and then apply MPC techniques.
This indirect data-driven control approach is known in the literature as subspace predictive control (SPC)~\cite{favoreel1999spc}.
The theoretical investigation of the relation between direct and indirect data-driven MPC is a largely open research problem.
First, we note that our results in Section~\ref{sec:nom} for nominal data-driven MPC are equally applicable to SPC if the model is identified exactly.
Recent work in~\cite{doerfler2021bridging} shows that the open-loop optimal control problem for data-driven MPC as considered in this paper is in fact a convex relaxation of the corresponding problem in SPC.
Additionally, direct data-driven and indirect model-based MPC have been compared for practical applications, e.g., in~\cite{carlet2020data}.
Generally, the closed-loop performance of either data-driven or model-based MPC depends on the accuracy of the involved ``model'' used to predict future trajectories.
While deriving tight error bounds in system identification based on a noisy data trajectory of finite length is a challenging problem, this paper provides a theoretical analysis of direct data-driven MPC for which the impact of noise on the prediction error is explicitly quantified, cf. the proof of Lemma~\ref{lem:robustStabLemma}.
In particular, the literature on SPC does not provide closed-loop guarantees under assumptions comparable to those in the present paper.
However, we conjecture that our proof of robust stability in Theorem~\ref{thm:practicalStabilityThm} can be adapted to SPC, assuming that a suitable bound on the identification error is available.
Another noteworthy feature of direct data-driven MPC is that online data updates can make the approach applicable to nonlinear systems, even providing closed-loop stability guarantees~\cite{berberich2021linearpart2}.
\end{remark}

%!TEX root = ./Main.tex
%%%%%%%%%%%%%%%%%%%%%%%%%%%%%%%%%%%%%%%%%%%%%%%%%%%%%%%%%%%%%%%%%%%%%%%%%%%%%%%
\section{Numerical example}
\label{sec:num}
In the following, we apply the robust data-driven MPC scheme presented in Section~\ref{sec:robust} to System~\eqref{eq:MinimalLTIsystem} with
\begin{align*}
A&=\begin{bmatrix}0.9749&-0.0135\\0.0004&0.9888\end{bmatrix},\>
B=10^{-4}\cdot\begin{bmatrix}0.041\\5.934\end{bmatrix},\\
C&=\begin{bmatrix}0&1\end{bmatrix},\>D=0.
\end{align*}
This system corresponds to the linearization of the nonlinear continuous stirred-tank reactor (CSTR) considered in~\cite{mayne2011tube} with linearization point $\begin{bmatrix}0.9831\\0.3918\end{bmatrix}$ and sampling time $0.5$.
Our goal is stabilization of the origin
%setpoint $(u^{\mathrm{s}},y^{\mathrm{s}})=(0.8,-0.0417)$, which is a feasible equilibrium for the above system dynamics (compare~\cite[Definition 3]{Stability&RobustnessTEC}), 
while satisfying input constraints $u_t\in\mathbb{U}=[-0.1,0.1]$ for $t\geq0$.
We assume that one noisy input-output trajectory $\{u_k^{\mathrm{d}},\tilde{y}_k^{\mathrm{d}}\}_{k=0}^{N-1}$ of the linearized CSTR is available with data length $N=200$.
This trajectory is generated by sampling the input uniformly from $u_k^{\mathrm{d}}\in\mathbb{U}$ and the output measurement noise affecting the data and the initial conditions uniformly from $\varepsilon_k^d\in[-\bar{\varepsilon},\bar{\varepsilon}]$ with $\bar{\varepsilon}=0.001$.

Figure~\ref{fig:CSTR} displays the closed-loop input-output trajectory resulting from the application of the robust data-driven MPC scheme (denoted by UCON, i.e., ``unconstrained''), compare Algorithm~\ref{alg:robustMPC}, to the linearized CSTR, where the design parameters are chosen as
\begin{align*}
&L=20,\>Q=I,\>R=10^{-2}I,\>\lambda_{\alpha}\bar{\varepsilon}=10^{-2},\>\frac{\lambda_{\sigma}}{\bar{\varepsilon}}=10^5.
\end{align*}
Note that the closed-loop input-output trajectory indeed converges close to the origin despite the noisy output measurements, i.e., the presented MPC scheme solves the control task.
Figure~\ref{fig:CSTR} also shows the closed loop resulting from the data-driven MPC scheme with terminal equality constraints (denoted by TEC), which was developed in~\cite{Stability&RobustnessTEC}, with the same parameters as above and omitting the non-convex constraint, compare Remark~\ref{rk:comparison_to_TEC}.
The input computed via this scheme is much more aggressive and fluctuating, showcasing the lack of robustness caused by terminal equality constraints.
For a quantitative comparison of the two MPC schemes, we compute for each closed-loop trajectory displayed in Figure~\ref{fig:CSTR} the cost function
\begin{align}\label{eq:cost_example}
\sum_{t=0}^{500}\lVert u_t\rVert_R^2+\lVert y_t\rVert_Q^2.
\end{align}
The closed-loop cost~\eqref{eq:cost_example} for the scheme in~\cite{Stability&RobustnessTEC} is $3.3\%$ larger than that of Algorithm~\ref{alg:robustMPC}.
Thus, dropping terminal equality constraints as we propose in this paper can not only avoid heavy input fluctuations but also leads to quantitative performance improvements.
To summarize, the MPC approach presented in this paper can be superior in practice if compared to existing approaches while at the same time possessing strong theoretical guarantees on closed-loop stability and robustness.

\begin{figure}
		\begin{center}
		\subfigure[Closed-loop input $u$]
		{\includegraphics[width=0.49\textwidth]{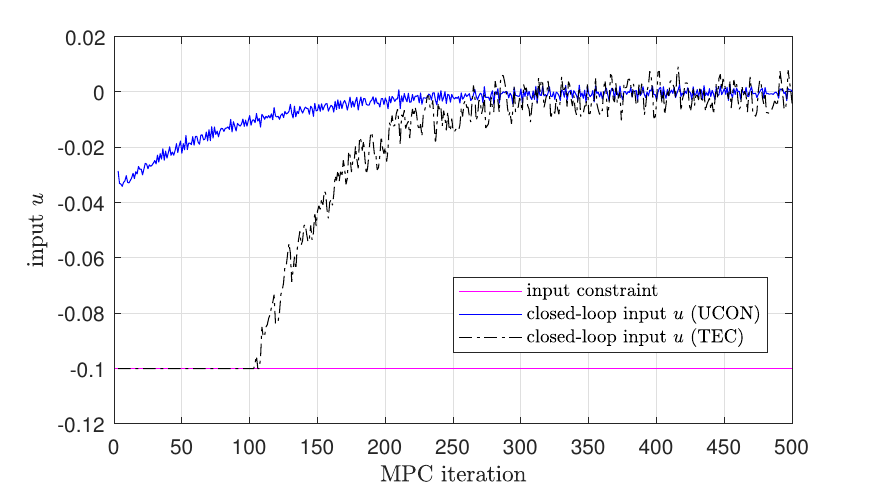}}
		\subfigure[Closed-loop output $y$]
		{\includegraphics[width=0.49\textwidth]{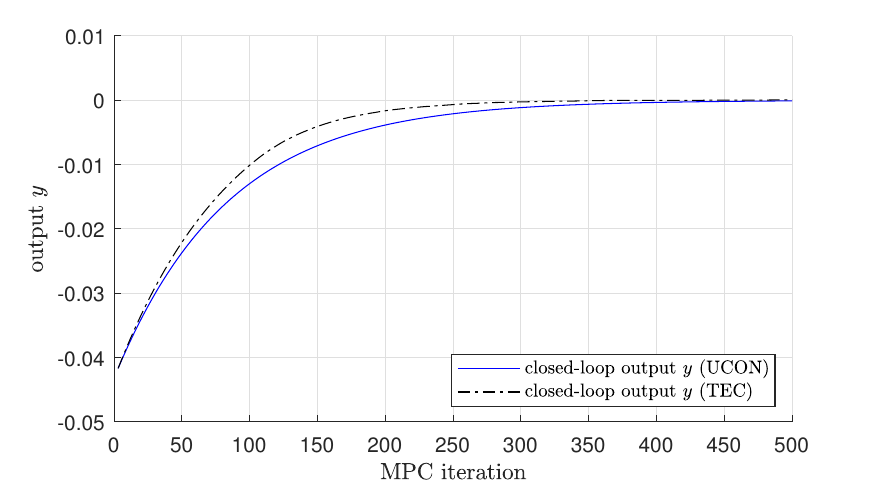}}
		\end{center}
		\caption{Closed-loop input and output, resulting from the application of the robust data-driven MPC scheme without terminal ingredients (UCON), compare Algorithm~\ref{alg:robustMPC}, and with terminal equality constraints (TEC), compare~\cite{Stability&RobustnessTEC}.}	\label{fig:CSTR}
\end{figure}
%!TEX root = ./Main.tex
%%%%%%%%%%%%%%%%%%%%%%%%%%%%%%%%%%%%%%%%%%%%%%%%%%%%%%%%%%%%%%%%%%%%%%%%%%%%%%%
\section{Conclusion}
\label{sec:conclusion}
In this paper, we analyzed closed-loop properties of a simple data-driven MPC scheme without stabilizing terminal ingredients.
This MPC scheme  does not require any model knowledge but only one input-output trajectory which may be affected by noise. 
Our main contribution is a proof of closed-loop practical exponential stability under this MPC scheme. 
In contrast to existing works on data-driven MPC with noisy data, we are not limited to open-loop robustness guarantees~\cite{DeePC4} and we require no terminal equality constraints~\cite{Stability&RobustnessTEC}, which can potentially deteriorate robustness.
Moreover, we illustrated the advantages of the considered data-driven MPC with a numerical example. 
Interesting topics for future research include an extension of the presented results to classes of nonlinear systems and an in-depth comparison of closed-loop properties in data-driven and model-based MPC.

%%%%%%%%%%%%%%%%%%%%%%%%%%%%%%%%%%%%%%%%%%%%%%%%%%%%%%%%%%%%%%%%%%%%%%%%%%%%%%%%
\bibliographystyle{ieeetran}  
\bibliography{Literature}  
%
%!TEX root = ./Main.tex
%%%%%%%%%%%%%%%%%%%%%%%%%%%%%%%%%%%%%%%%%%%%%%%%%%%%%%%%%%%%%%%%%%%%%%%%%%%%%%%
\appendix
%\section*{Appendix}
\subsection{Proof of Lemma \ref{lem:robustStabLemma}}
\label{app:proof}
The proof is structured as follows.
In Part I, the one-step prediction error in the state is bounded based on the noise level $\bar{\varepsilon}$. In Part II, after one step of Algorithm \ref{alg:robustMPC} at time $t+1$, a new, shortened candidate is constructed based on the previous solution. %, which is shown to be feasible. 
Part III and Part IV consist of using the previously established prediction error bound and the new shortened candidate to establish continuity-like properties of the value function $J_L^*$ and the IOSS Lyapunov function $W$, respectively.
%, where we show that the deviation of these functions is bounded depending the noise level $\bar{\varepsilon}$.
The desired continuity-like property of the Lyapunov candidate $Y_L$ follows directly in Part~V.\\
%It is then straightforward to derive a similar continuity-like property of the Lyapunov candidate function $Y_L$ in Part V.\\
%\begin{proof}
\textbf{Part I } 
First, denote the nominal one-step prediction of the extended state $\xi$ from the previously optimal trajectory a time $t$ as
	\begin{align} \label{eq:NominalOneStepStatePrediction}
			&\hat{\xi}_1^*(t) \coloneqq 
			\begin{bmatrix}
			\begin{bmatrix} 0_{m \times m} \ I_{lm} \ 0_{lm \times m(L-1)} \end{bmatrix}
			H_{L+l}(u^{\mathrm{d}})\alpha^*(t)\\
			\begin{bmatrix} 0_{p \times p} \ I_{lp} \ 0_{lp \times p(L-1)} \end{bmatrix}
			H_{L+l}(y^{\mathrm{d}})\alpha^*(t)
			\end{bmatrix}\\
			&\stackrel{\eqref{eq:robustDDUNCMPC_Dynamics}, \eqref{eq:robustDDUNCMPC_IC}}{=}\begin{bmatrix}
				u_{[t+1-l,t-1]}\\
				\bar{u}_0^*(t)\\
				\tilde{y}_{[t+1-l,t-1]}\\
				\bar{y}^*_0(t)
			\end{bmatrix}\nonumber\\
			&+ 
			\begin{bmatrix}
				0_{lm \times 1}\\
				- \begin{bmatrix} 0_{lp \times p} \ I_{lp} \ 0_{lp \times p(L-1)} \end{bmatrix}
				H_{l+L}(\varepsilon^\mathrm{d}) \alpha^*(t) + \sigma_{[-l+1,0]}^*(t)
			\end{bmatrix}. \nonumber
	\end{align}
	Next, define the one-step state prediction error $\delta^\xi_1 (t)$ as
	\begin{align} 
%\label{eq:OneStepStatePredError}
\label{eq:OneStepStatePredErrorExpanded}
		&\delta^\xi_1 (t) \coloneqq \xi_{t+1} - \hat{\xi}^*_1 (t)
%		=
%		\begin{bmatrix}
%			u_{[t+1-l,t-1]}\\
%			\bar{u}_0^*(t)\\
%			y_{[t+1-l,t-1]}\\
%			y_t
%		\end{bmatrix}
%		-
%		\begin{bmatrix}
%			u_{[t+1-l,t-1]}\\
%			\bar{u}_0^*(t)\\
%			\tilde{y}_{[t+1-l,t-1]}\\
%			\bar{y}_0^*(t)
%		\end{bmatrix},
%	\end{align}
%	By expanding the definition \eqref{eq:OneStepStatePredError}, we obtain
%	\begin{align} 
%		&\delta^\xi_1(t) =
=
		\begin{bmatrix}
			0_{lm \times 1}\\
			- \varepsilon_{[t+1-l,t-1]}\\
			y_t - \bar{y}_0^*(t)
		\end{bmatrix}\\
		& + 
		\begin{bmatrix}
			0_{lm \times 1}\\ 
			\begin{bmatrix}
			0_{lp\times p}& 	I_{lp} & 0_{lp \times p(L-1)}
			\end{bmatrix}
			H_{l+L}(\varepsilon^\mathrm{d}) \alpha^*(t)
			- \sigma_{[-l+1,0]}^*(t)
		\end{bmatrix}. \nonumber
	\end{align}
%	which shows that the one-step prediction error is
%	\begin{align} \label{eq:PredErrorOnlyInLastEntry}
%		\delta^\xi_1 (t) =
%		\begin{bmatrix}
%			0_{lm+(l-1)p}\\
%			\tilde{y}_t - \bar{y}_0^*(t)
%		\end{bmatrix}.
%	\end{align}
%	The perturbed output $\tilde{y}_t$ is given by
%	\begin{align} \label{eq:PredErrorActualOutput}
%		\tilde{y}_t \overset{\eqref{eq:xiDynamics}}{=} \tilde{C} \xi_t + \tilde{D} \bar{u}_0^*(t) + \varepsilon_t.
%	\end{align}
	The predicted output $\bar{y}_0^*(t)$ by definition is a part of the predicted trajectory
	\begin{align}
		\bar{y}_0^*(t) = \begin{bmatrix} 0_{p \times pl} \ I_p \ 0_{p \times p(L-1)} \end{bmatrix} \bar{y}^*(t),
	\end{align}
	where the predicted trajectory $\bar{y}^*(t)$ is given in the MPC problem \eqref{eq:robustDDUNCMPC}.
	Moreover, we obtain
	\begin{align} \label{eq:PredErrorPredOutput}
		\begin{split}
			%		\bar{y}_0^*(t) &= \begin{bmatrix} 0_{p \times l} \ I_p \ 0_{p \times L-1} \end{bmatrix} H_{l+L}(y^\mathrm{d}) \alpha^*(t)\\
			%		&+ \begin{bmatrix} 0_{p \times l} \ I_p \ 0_{p \times L-1} \end{bmatrix} \left( H_{l+L}(\varepsilon^\mathrm{d}) \alpha^*(t) - \sigma^*(t) \right).
			\bar{y}^*(t) + \sigma^*(t) - H_{l+L}(\varepsilon^\mathrm{d}) \alpha^*(t) = H_{l+L}(y^\mathrm{d}) \alpha^*(t).
		\end{split}
	\end{align}
	%	Note that the data-driven part $\begin{bmatrix} 0_{p \times l} \ I_p \ 0_{p \times L-1} \end{bmatrix} H_{l+L}(y^\mathrm{d}) \alpha^*(t)$ corresponds to a perfect prediction with perturbed initial condition, i.\,e.,
	Note that $H_{l+L}(y^\mathrm{d}) \alpha^*(t)$ is a trajectory of the system according to Lemma \ref{lem:FundamentalLemma}. 
	%	The initial condition of this trajectory is set by the constraints \eqref{eq:robustDDUNCMPC_Dynamics}, \eqref{eq:robustDDUNCMPC_IC} as
	We define the following variable as an initial condition to the considered trajectory
	\begin{align} \label{eq:PredErrorIC}
			&\hat{\xi}_0(t) =\\\nonumber
			&\begin{bmatrix}
				u_{[t-l,t-1]}\\
				\tilde{y}_{[t-l,t-1]} + \sigma_{[-l,-1]}^*(t) - 
				\begin{bmatrix}
					I_{lp} \ 0_{lp \times Lp}
				\end{bmatrix}
				H_{l+L}(\varepsilon^\mathrm{d}) \alpha^*(t)
			\end{bmatrix}.
	\end{align}
	Using equivalence of data-driven and state space models (cf. Lemma \ref{lem:FundamentalLemma}), we obtain
	\begin{align}
		[0_{p\times lp} \ I_p \ 0_{p\times(L-1)p}]H_{l+L}(y^d)\alpha^*(t)=\tilde{C}\hat{\xi}_0(t)+\tilde{D}\bar{u}_0^*(t).
	\end{align}
		Combining this condition with~\eqref{eq:PredErrorPredOutput}, we obtain the following expression for the predicted output $\bar{y}_0^*(t)$:
		\begin{align}
%		\begin{split}
			\bar{y}_0^*(t) =& \tilde{C} \hat{\xi}_0(t) + \tilde{D} \bar{u}_0^*(t) - \sigma_0^*(t)\\
			&+[0_{p\times lp} \ I_p \ 0_{p\times (L-1)p}] H_{l+L}(\varepsilon^d)\alpha^*(t).\nonumber
%		\end{split}
	\end{align}
	Using \eqref{eq:PredErrorIC} yields the following expression for the predicted output $\bar{y}_0^*(t)$: 
	\begin{align} \label{eq:PredErrorPerfectPredictionPart}
			%		\begin{bmatrix} 0_{p \times l} \ I_p \ 0_{p \times (L-1)} \end{bmatrix} H_{l+L}(y^\mathrm{d}) \alpha^*(t)
			%	    \bar{y}_0^*(t) =& \tilde{C} \bar{\xi}_0^*(t) + \tilde{D} \bar{u}_0^*(t)\\
			\bar{y}_0^*(t) =& \tilde{C} \left( \xi_t + 
			\begin{bmatrix}
				0_{lm \times 1}\\
				\varepsilon_{[t-l,t-1]}
			\end{bmatrix} 
			\right) + \tilde{D} \bar{u}_0^*(t)\\
			&+ \tilde{C} 
			\begin{bmatrix}
				0_{lm \times 1}\\
				\sigma_{[-l,-1]}^*(t) - 
				\begin{bmatrix}
					I_{lp} \ 0_{lp \times Lp}
				\end{bmatrix}
				H_{l+L}(\varepsilon^\mathrm{d}) \alpha^*(t)
			\end{bmatrix}\nonumber\\
			&+ \begin{bmatrix} 0_{p \times lp} \ I_p \ 0_{p \times (L-1)p} \end{bmatrix} H_{l+L}(\varepsilon^\mathrm{d}) \alpha^*(t) - \sigma_0^*(t) .\nonumber
	\end{align}
	Therefore, combining~\eqref{eq:PredErrorPerfectPredictionPart} with~\eqref{eq:xiDynamics},
%	using \eqref{eq:PredErrorPredOutput} and \eqref{eq:PredErrorPerfectPredictionPart}, 
	the output prediction error is equal to
	\begin{align} \label{eq:PredErrorOutputError}
%		\begin{split}
			&y_t - \bar{y}_0^*(t)\\
			=&	- \tilde{C} 
			\begin{bmatrix}
				0_{lm \times 1}\\
				\varepsilon_{[t-l,t-1]} +
				\sigma_{[-l,-1]}^*(t) - 
				\begin{bmatrix}
					I_{lp} \ 0_{lp \times Lp}
				\end{bmatrix}
				H_{l+L}(\varepsilon^\mathrm{d}) \alpha^*(t)
			\end{bmatrix} \nonumber \\
			&- \begin{bmatrix} 0_{p \times lp} \ I_p \ 0_{p \times  (L-1)p} \end{bmatrix} H_{l+L}(\varepsilon^\mathrm{d}) \alpha^*(t) + \sigma_0^*(t). \nonumber
%		\end{split}
	\end{align}
	In the following, we derive bounds for the terms used in~\eqref{eq:OneStepStatePredErrorExpanded} and \eqref{eq:PredErrorOutputError}.
	The following derivation will repeatedly use the equivalence property of norms \cite[p. 72]{MatrixComputations}, i.e.,  $\|d\|_2\leq \sqrt{k}\|d\|_\infty$ for any $d\in\mathbb{R}^k$, to apply the noisy bound $\|\varepsilon\|_\infty\leq \bar{\varepsilon}$.
	To derive bounds on $\sigma^*(t)$ and $\alpha^*(t)$, we make use the fact that $Y_L(\tilde{\xi}_t) \leq \bar{Y}$ and therefore $J_L^*(\tilde{\xi}_t) \leq \bar{Y}$, which implies $ \dfrac{\lambda_\sigma}{\bar{\varepsilon}} \| \sigma^*(t) \|_2^2 \leq \bar{Y}$ using the value function $J_L^*$ in \eqref{eq:robustDDUNCMPC}. Hence, we obtain 
	\begin{align} \label{eq:PredErrorSigmaBound}
		%\begin{split}
			%\| \sigma_{[-l,-1]}^*(t) \|_2 &\leq
			 \| \sigma^*(t) \|_2 \leq \sqrt{\frac{\bar{Y} \bar{\varepsilon}}{\lambda_\sigma}}.%,\\
			%\| \sigma_0^*(t) \|_2 &\leq \| \sigma^*(t) \|_2.
		%\end{split}
	\end{align}
	Using the same arguments, it can be shown that
	%	\begin{align} \label{eq:PredErrorAlphaBound}
	%		\| \alpha^*(t) \|_2 \leq \sqrt{\frac{\bar{Y}}{\lambda_\alpha}}.
	%	\end{align}
	\begin{align} \label{eq:PredErrorAlphaBound}
		\bar{\varepsilon}\| \alpha^*(t) \|_2 \leq \sqrt{\frac{\bar{\varepsilon}\bar{Y}}{\lambda_\alpha }}.
	\end{align}
The terms in~\eqref{eq:PredErrorOutputError} regarding the noisy data can be bounded using
%	For the part in \eqref{eq:PredErrorOutputError} generated by noisy data, we have
	\begin{align} \label{eq:PredErrorNoisyDataPart2}
%		\begin{split}
			&\| \begin{bmatrix} 0_{p \times lp} \ I_p \ 0_{p \times (L-1)p} \end{bmatrix} H_{l+L}(\varepsilon^\mathrm{d}) \alpha^*(t) \|_2\\
			\leq &\| \begin{bmatrix} 0_{p \times lp} \ I_p \ 0_{p \times (L-1)p} \end{bmatrix} H_{l+L}(\varepsilon^\mathrm{d}) \|_2 \| \alpha^*(t) \|_2.\nonumber
%		\end{split}
	\end{align}
The noise-dependent term in \eqref{eq:PredErrorNoisyDataPart2} is bounded as
%	Consider
	\begin{align}
		\| \begin{bmatrix} 0_{p \times lp} \ I_p \ 0_{p \times (L-1)p} \end{bmatrix} H_{l+L}(\varepsilon^\mathrm{d}) \|_\infty \leq \bar{\varepsilon}N_L,%$ (N-L-l+1),
	\end{align}
	where $N_L:=(N-L-l+1)$ corresponds to the number of columns in the Hankel matrix.
	%by bounding the maximum absolute row sum.
	Using equivalence of norms, we obtain
	\begin{align} \label{eq:PredErrorNoisyDataPart}
			&\| \begin{bmatrix} 0_{p \times lp} \ I_p \ 0_{p \times (L-1)p} \end{bmatrix} H_{l+L}(\varepsilon^\mathrm{d}) \|_2\\\nonumber
			&\leq \sqrt{p} \| \begin{bmatrix} 0_{p \times lp} \ I_p \ 0_{p \times (L-1)p} \end{bmatrix} H_{l+L}(\varepsilon^\mathrm{d}) \|_\infty\leq \bar{\varepsilon} \sqrt{p}N_L.% (N-L-l+1).
	\end{align}
	Similarly, the other noisy data matrix in~\eqref{eq:PredErrorOutputError} satisfies
		\begin{align}
		\label{eq:PredErrorNoisyDataPart_lp}
			&\| \begin{bmatrix} I_{lp} \ 0_{lp \times Lp} \end{bmatrix} H_{l+L}(\varepsilon^\mathrm{d}) \|_2\nonumber\\
			&\leq \sqrt{lp} \| \begin{bmatrix} I_{lp} \ 0_{lp \times Lp} \end{bmatrix} H_{l+L}(\varepsilon^\mathrm{d}) \|_\infty\leq \bar{\varepsilon}N_L \sqrt{lp} .%(N-L-l+1).
	\end{align}
Using the bounds~\eqref{eq:PredErrorSigmaBound}, \eqref{eq:PredErrorAlphaBound}, \eqref{eq:PredErrorNoisyDataPart2}, \eqref{eq:PredErrorNoisyDataPart} in~\eqref{eq:PredErrorOutputError}, we arrive at the following bound for the output-prediction error
	\begin{align} \label{eq:OutputPredErrorNormBound}
		&\| y_t - \bar{y}_0^*(t) \|_2 \\
		\leq 		& \bar{\varepsilon}\|\tilde{C}\|\sqrt{pl}
		+\sqrt{\bar{\varepsilon}}\sqrt{\bar{Y}}\left(\dfrac{1+\|\tilde{C}\|_2}{\sqrt{\lambda_\sigma}}+\dfrac{N_L(\sqrt{p}+\|\tilde{C}\|_2 \sqrt{pl})}{\sqrt{\lambda_\alpha}}\right).\nonumber
	\end{align}
Using~\eqref{eq:PredErrorSigmaBound}, \eqref{eq:PredErrorAlphaBound}, and a bound analogous to~\eqref{eq:PredErrorNoisyDataPart_lp}, the one-step error bound~\eqref{eq:OneStepStatePredErrorExpanded} can be simplified to
	\begin{align}
		\| \delta_1^\xi(t) \|_2 \leq&\| y_t - \bar{y}_0^*(t) \|_2+ \bar{\varepsilon} \sqrt{p(l-1)} \\
		&+ \sqrt{\bar{\varepsilon}}\sqrt{\bar{Y}} \left( \frac{1}{\sqrt{\lambda_\sigma}} +  \frac{\sqrt{lp}N_L}{\sqrt{\lambda_\alpha}} \right). \nonumber
	\end{align}
%	\begin{align}
%		\| \delta_1^\xi(t) \|_2 \JK{\leq}& \bar{\varepsilon} \sqrt{lp-1} + \sqrt{\bar{\varepsilon}} \left( \sqrt{\frac{\bar{Y}}{\lambda_\sigma}} + \sqrt{lp} \sqrt{\frac{\bar{Y}}{\lambda_\alpha}} \right)\\
%		&+ \| y_t - \bar{y}_0^*(t) \|_2. \nonumber
%	\end{align}
Using condition~\eqref{eq:OutputPredErrorNormBound}, this bound reduces to
	\begin{align} \label{eq:OneStepPredictionErrorNormBound}
		\| \delta^\xi_1 (t) \|_2 \leq b_\xi(\bar{\varepsilon}),
	\end{align}
	where the class $\mathcal{K}_\infty$-function $b_\xi(\bar{\varepsilon})$ is given by
	\begin{align} \label{eq:b_xi}
			&b_\xi(\bar{\varepsilon}) \coloneqq \bar{\varepsilon} (\|\tilde{C}\|_2 \sqrt{lp} + \sqrt{p(l-1)} )\\
			&+\sqrt{\bar{\varepsilon}}\sqrt{\bar{Y}}\left(
			 \dfrac{2+\|\tilde{C}\|_2}{\sqrt{\lambda_\sigma}}
			 +\dfrac{N_L(\sqrt{p}+(\|\tilde{C}\|_2+1) \sqrt{pl})}{\sqrt{\lambda_\alpha}}\right).\nonumber
			 %\\
			%+& \sqrt{\bar{\varepsilon}} \bigg( \| \tilde{C} \|_2 \left( \sqrt{lp} (N-l+1) \sqrt{\frac{\bar{Y}}{\lambda_\alpha}} + \sqrt{\frac{\bar{Y}}{\lambda_\sigma}} \right) \nonumber\\
%			&+ \sqrt{p} (N-l+1) \sqrt{\frac{\bar{Y}}{\lambda_\alpha}} + 2 \sqrt{\frac{\bar{Y}}{\lambda_\sigma}} + \sqrt{lp} \sqrt{\frac{\bar{Y}}{\lambda_\alpha}} \bigg). \nonumber
	\end{align}
%	Using the error equation \eqref{eq:OneStepStatePredErrorExpanded} together with the bounds \eqref{eq:PredErrorOutputError}, \eqref{eq:PredErrorOutputErrorBound}, \eqref{eq:PredErrorNoisyDataPart}, \eqref{eq:PredErrorSigmaBound}, \eqref{eq:PredErrorAlphaBound}, \eqref{eq:OutputPredErrorNormBound} and $(N-L-l+1) \leq (N-l+1)$ yields
\textbf{Part II }
	First, we construct a new, feasible input candidate at time $t+1$ of shortened length $L-1$ by resuming the previous sequence from time $t$, i.\,e.
	\begin{align} \label{eq:NewCandInput}
		\bar{u}_{[-l,L-2]}(t+1) = \bar{u}^*_{[-l+1,L-1]}(t),
	\end{align}
	and we proceed by finding the corresponding $\alpha$ according to Lemma \ref{lem:FundamentalLemma}. Consider the Hankel matrix
	\begin{align}
		H_{ux} \coloneqq \begin{bmatrix}
			H_{l+L-1}(u^\mathrm{d})\\
			H_1(x^\mathrm{d}_{[0,N-L-l+1]})
		\end{bmatrix},
	\end{align}
	which has full row rank since $u^\mathrm{d}$ is persistently exciting of order $n+L$ according to \cite[Corollary 2, (iii)]{Willems05}. The sequence $x^\mathrm{d}$ is uniquely determined by the sequences $u^\mathrm{d}, y^\mathrm{d}$ since System \eqref{eq:MinimalLTIsystem} is observable and $l\geq\underline{l}$. 
	At time $t+1$, the last $l$ I/O measurements $\xi_{t+1}$ invoke a unique initial condition $x_{t+1-l}$. Together with the new input sequence \eqref{eq:NewCandInput}, we obtain with Lemma \ref{lem:FundamentalLemma}
	\begin{align}
		\begin{bmatrix}
			\bar{u}_{[-l,L-2]}(t+1)\\
			x_{t+1-l}
		\end{bmatrix}
		= H_{ux} \bar{\alpha}(t+1),
	\end{align}
	where $\bar{\alpha}$ represents the new candidate vector $\alpha$ from Lemma \ref{lem:FundamentalLemma}.
	Since $H_{ux}$ has full row rank, there exists the right-inverse
	\begin{align}
		H_{ux}^\dagger = H_{ux}^\top \left(H_{ux} H_{ux}^\top \right)^{-1},
	\end{align}
	which can be used to compute the vector $\bar{\alpha}(t+1)$ as
	\begin{align} \label{eq:NewCandidateAlpha}
		\bar{\alpha}(t+1) = H_{ux}^\dagger
		\begin{bmatrix}
			\bar{u}_{[-l,L-2]}(t+1)\\
			x_{t+1-l}
		\end{bmatrix}.
	\end{align}
	The output candidate is chosen as the corresponding nominal output trajectory, i.\,e.
	\begin{align} \label{eq:NewOutputCandidate}
		\bar{y}_{[-l,L-2]}(t+1) = H_{L+l-1}(y^\mathrm{d}) \bar{\alpha}(t+1).
	\end{align}
	In order to satisfy the constraints \eqref{eq:robustDDUNCMPC_Dynamics} and \eqref{eq:robustDDUNCMPC_IC}, the slack variable candidate is chosen as
	\begin{align} \label{eq:NewSlackCandidate}
		\bar{\sigma}_{[-l,L-2]}(t+1) = H_{L+l-1}(\varepsilon^\mathrm{d}) \bar{\alpha}(t+1) - \begin{bmatrix}
			\varepsilon_{[t+1-l,t]}\\
			0_{p(L-1) \times 1}
		\end{bmatrix}.
	\end{align}
	\textbf{Part III} 
%	\newline
	For the next step, a bound of the deviation between the previous candidate trajectory and the new shortened one is shown. The new output candidate solution is chosen to be a trajectory of the nominal system in \eqref{eq:NewOutputCandidate}. Therefore, the output candidate trajectory can be equivalently rewritten using the dynamics of the extended system \eqref{eq:xiDynamics} as
	\begin{align} \label{eq:ModelOutputtPlusOne}
			\bar{y}_k(t+1) =& \tilde{C} \tilde{A}^k \xi_{t+1}
			+ \tilde{C} \sum_{j = 0}^{k-1} \tilde{A}^j \tilde{B} \bar{u}_j(t+1)
			+ \tilde{D}\bar{u}_k(t+1),
	\end{align}
	for $k \in\mathbb{I}_{[0,L-2]}$. Using the previous optimal solution from \eqref{eq:robustDDUNCMPC_Dynamics} and \eqref{eq:NominalOneStepStatePrediction}, we similarly obtain
	\begin{align} \label{eq:ModelOutputt}
			\bar{y}_{k+1}^*(t) =& \tilde{C} \tilde{A}^k \hat{\xi}_1^*(t) + \tilde{C} \sum_{j = 0}^{k-1} \tilde{A}^j \tilde{B} \bar{u}_{j+1}^*(t) + \tilde{D} \bar{u}_{k+1}^*(t)\nonumber\\
			&+ \begin{bmatrix} 0_{p \times pk} \ I_p \ 0_{p \times p(l+L-k-1)} \end{bmatrix} H_{L+l}(\varepsilon^\mathrm{d}) \alpha^*(t)\nonumber\\
			&- \sigma^*_{k+1}(t),
	\end{align}
	for $k \in\mathbb{I}_{[0,L-2]}$. Subtracting the old optimal solution~\eqref{eq:ModelOutputt} from the new one \eqref{eq:ModelOutputtPlusOne}, and using that the input candidate was shifted \eqref{eq:NewCandInput}, the error between the two output trajectories is given by
	\begin{align}
	\label{eq:def_delta_y}
			\underbrace{\bar{y}_k(t+1) - \bar{y}_{k+1}^*(t)}_{\eqqcolon \delta_k^y(t+1)} \stackrel{\eqref{eq:OneStepStatePredErrorExpanded}}{=}& \tilde{C} \tilde{A}^k \delta_1^\xi(t)\\
			- \begin{bmatrix} 0_{p \times pk} \ I_p \ 0_{p \times p(l+L-k-1)} \end{bmatrix}& H_{L+l}(\varepsilon^\mathrm{d}) \alpha^*(t) + \sigma^*_{k+1}(t).\nonumber
	\end{align}
The $Q$-weighted norm of the output error $\delta_k^y(t+1)$ can be bounded using
	\begin{align}
			&\| \delta_k^y(t+1) \|_Q\\
			 \leq& \| \tilde{C} \tilde{A}^k \|_Q \| \delta_1^\xi(t) \|_2
			+ \sqrt{\lambda_{\max}(Q)} \big( \| \sigma_{k+1}^*(t) \|_2\nonumber\\
			&+  \| \begin{bmatrix} 0_{p \times pk} \ I_p \ 0_{p \times p(l+L-k-1)} \end{bmatrix}H_{L+l}(\varepsilon^\mathrm{d}) \|_2 \| \alpha^*(t) \|_2 \big).\nonumber
	\end{align}
	Let $c_1, \rho_1 > 0$ such that $\| \tilde{C} \tilde{A}^{k} \|_Q \leq c_1 \rho_1^k$. 
%	Squaring this bound leads to $\| \tilde{C} \tilde{A}^{k} \|_Q^2 \leq c_1^2 \rho_1^{2k}$. 
%	Using this bound on $\| \tilde{C} \tilde{A}^{k} \|_Q$ and the slack variable candidate \eqref{eq:NewSlackCandidate}, the output prediction error can be bounded as
Using this bound, conditions~\eqref{eq:PredErrorSigmaBound}, \eqref{eq:PredErrorAlphaBound}, \eqref{eq:OneStepPredictionErrorNormBound}, and a bound analogous to~\eqref{eq:PredErrorNoisyDataPart}, the output prediction error can be bounded as
	\begin{align} \label{eq:OutputPredErrorBound1}
\| \delta_k^y(t+1) \|_Q \leq& c_1 \rho_1^k b_\xi(\bar{\varepsilon})\\
& +\sqrt{\bar{\varepsilon}}\underbrace{\sqrt{\bar{Y}}\sqrt{\lambda_{\max}(Q)}\left( \dfrac{1}{\sqrt{\lambda_\sigma}} +\dfrac{N_L \sqrt{lp}}{\sqrt{\lambda_{\alpha}}}\right)}_{\eqqcolon d_1}.\nonumber
	\end{align}
	Squaring the bound \eqref{eq:OutputPredErrorBound1} and using the inequality $(a+b)^2 \leq 2a^2+2b^2$ leads to
	\begin{align}
\label{eq:OutputPredErrorBound2_square}
\| \delta_k^y(t+1) \|_Q^2 \leq 2 c_1^2 \rho_1^{2k} b_\xi(\bar{\varepsilon})^2 + 2 \bar{\varepsilon} d_1^2.
	\end{align}
	Summing up over $L-1$ steps leads to
	\begin{align} \label{eq:DeltaYBound}
			\sum_{k = 0}^{L-2} \| \delta_k^y(t+1) \|_Q \stackrel{\eqref{eq:OutputPredErrorBound1}}{\leq}& \sqrt{\bar{\varepsilon}} (L-1) d_1+ b_\xi(\bar{\varepsilon}) c_1 \sum_{k=0}^{L-2}\rho_1^{k},\\
 \label{eq:DeltaYSquaredBound}
			\sum_{k = 0}^{L-2} \| \delta_k^y(t+1) \|_Q^2 \stackrel{\eqref{eq:OutputPredErrorBound2_square}}{\leq}& \bar{\varepsilon} \underbrace{2(L-1) d_1^2}_{\eqqcolon d_2}
			 + b_\xi(\bar{\varepsilon})^2 \underbrace{2 c_1^2 \sum_{k=0}^{L-2}\rho_1^{2k}}_{\eqqcolon d_3}.
	\end{align}
	For the next step, a norm bound on the new candidate $\bar{\alpha}(t+1)$ as in \eqref{eq:NewCandidateAlpha} is shown. Using the lag of the system, it can be shown that
	\begin{align}
		x_{t+1-l} = \begin{bmatrix} M_1 \ \mathcal{O}_l^{-1} \end{bmatrix} \xi_{t+1},
	\end{align}
	where $\mathcal{O}_l^{-1}$ denotes the left-inverse of the observability matrix $\mathcal{O}_l$ in~\eqref{eq:obs_matrix}, and the matrix $M_1$ contains the system matrices $A,B,C,D$. Using the dynamics of the extended system \eqref{eq:xiDynamics}, it can be shown that
	\begin{align}
	\label{eq:app_x_t_l}
		x_{t+1-l} = \begin{bmatrix} M_1 \ \mathcal{O}_l^{-1} \end{bmatrix} (\tilde{A} \tilde{\xi}_t - \tilde{A} \begin{bmatrix}
			0_{lm \times 1}\\
			\varepsilon_{[t-l,t-1]}
		\end{bmatrix}
		+ \tilde{B} \bar{u}_0^*(t) ).
	\end{align}
	Since we assume $J_L^*(\tilde{\xi}_t)\leq Y_L(\tilde{\xi}_t) \leq \bar{Y}$, it follows that %$\| \bar{u}_{[-l,L-1]}^*(t) \|_R^2 \overset{\eqref{eq:robustDDUNCMPC_Cost}}{\leq} \bar{Y}$, and therefore
	\begin{align}
	\label{eq:app_u_opt_bound}
		\| \bar{u}_{[-l,L-1]}^*(t) \|_2^2 \leq \frac{\bar{Y}}{\lambda_{\min}(R)}.
	\end{align}
	 Applying~\eqref{eq:app_u_opt_bound} and the lower bound in~\eqref{eq:PSTBoundOnYL} to~\eqref{eq:app_x_t_l}, we obtain the following bound
	\begin{align}
	\label{eq:app_x_t_l_2}
			&\| x_{t+1-l} \|_2 \\\nonumber
			\leq &\underbrace{c_\xi \bigg(  \|\tilde{A} \|_2\left(\dfrac{\sqrt{\bar{Y}}}{\sqrt{\epsilon_{\mathrm{o}}}} + \sqrt{lp} \bar{\varepsilon}\right)
			+ \| \tilde{B} \|_2 \frac{\sqrt{\bar{Y}}}{\sqrt{\lambda_{\min}(R)}} \bigg)}_{\eqqcolon c_x}
	\end{align}
	with $c_\xi \coloneqq \| \begin{bmatrix} M_1 \ \mathcal{O}_l^{-1} \end{bmatrix} \|_2$.
	Therefore, the norm of the candidate $\bar{\alpha}(t+1)$ can be bounded as
	\begin{align}
	\label{eq:bound_alpha_cand}	
				\| \bar{\alpha}(t+1) \|_2^2 \stackrel{\eqref{eq:NewCandidateAlpha}}{\leq}& \| H_{ux}^\dagger \|_2^2 \left( \| \bar{u}_{[-l+1,L-1]}^*(t) \|_2^2 + \| x_{t+1-l} \|_2^2 \right) \nonumber\\
			\stackrel{\eqref{eq:app_u_opt_bound},\eqref{eq:app_x_t_l_2}}{\leq}& \underbrace{\| H_{ux}^\dagger \|_2^2 \left( \frac{\bar{Y}}{\lambda_{\min}(R)} + c_x^2 \right)}_{\eqqcolon c_\alpha^2}.
	\end{align}
Using properties of the matrix norm (cf.~\cite{MatrixComputations}), we have
	\begin{align}
		\| H_{l+L-1} (\varepsilon^\mathrm{d}) \|_2 \leq \overline{\epsilon}\underbrace{\sqrt{(l+L-1)p}\sqrt{(N-L-l+2)p}}_{=:b_H}.
		\end{align}
Thus, the candidate slack variable $\bar{\sigma}(t+1)$ in \eqref{eq:NewSlackCandidate} satisfies the following bound
	\begin{align}
	\label{eq:bound_sigma_cand}
			\| \bar{\sigma}(t+1) \|_2 \overset{\eqref{eq:NewSlackCandidate}}{\leq}& \| H_{l+L-1} (\varepsilon^\mathrm{d}) \|_2 \| \bar{\alpha}(t+1) \|_2 + \| \varepsilon_{[t+1-l,t]} \|_2 \nonumber\\
			\leq& \bar{\varepsilon} \underbrace{\left( b_H c_\alpha +\sqrt{ lp} \right)}_{\eqqcolon b_\sigma}.
	\end{align}
Finally, using the fact that the candidate trajectory is a feasible solution to~\eqref{eq:robustDDUNCMPC}, we can upper bound the value function using
%	Thus, a bound on the new shortened candidate cost $J_{L-1}(\tilde{\xi}_{t+1}) \geq J_{L-1}^*(\tilde{\xi}_{t+1})$ can be shown as follows
	\begin{align}
			J_{L-1}^*(\tilde{\xi}_{t+1}) \leq& \sum_{k = 0}^{L-2} \| \bar{u}_k(t+1) \|_R^2 + \| \bar{y}_k(t+1) \|_Q^2\\\nonumber
			&+ \lambda_\alpha \bar{\varepsilon} \| \bar{\alpha}(t+1) \|_2^2 + \lambda_\sigma \frac{1}{\bar{\varepsilon}} \| \bar{\sigma}(t+1) \|_2^2\\\nonumber
			\overset{\eqref{eq:NewCandInput},\eqref{eq:def_delta_y},\eqref{eq:bound_alpha_cand},\eqref{eq:bound_sigma_cand}}{\leq}& \underbrace{\sum_{k = 0}^{L-2} \| \bar{u}_{k+1}^*(t) \|_R^2 + \| \bar{y}_{k+1}^*(t) \|_Q^2}_{= J_L^*(\tilde{\xi}_t) - \ell(\bar{u}_0^*(t), \bar{y}_0^*(t))}\\\nonumber
			&+ 2 \sum_{k = 0}^{L-2} \| \bar{y}_{k+1}^*(t) \|_Q \| \delta_{k}^y(t+1) \|_Q\\\nonumber
			&+ \sum_{k = 0}^{L-2} \| \delta_{k}^y(t+1) \|_Q^2+ \lambda_\alpha \bar{\varepsilon} c_\alpha^2 + \lambda_\sigma \bar{\varepsilon} b_\sigma^2\\\nonumber
			\overset{\eqref{eq:DeltaYBound}, \eqref{eq:DeltaYSquaredBound}}{\leq}& J_L^*(\tilde{\xi}_t) - \ell(\bar{u}_0^*(t), \bar{y}_0^*(t)) + \alpha_1(\bar{\varepsilon}),
	\end{align}
	with
	\begin{align}
			\alpha_1(\bar{\varepsilon}) \coloneqq& \bar{\varepsilon} d_2 + b_\xi(\bar{\varepsilon})^2 d_3 
			 + \lambda_\alpha \bar{\varepsilon} c_\alpha^2 + \lambda_\sigma \bar{\varepsilon} b_\sigma^2\\
			& + 2 \sqrt{\bar{Y}} b_\xi(\bar{\varepsilon}) c_1 \sum_{k=0}^{L-2}\rho_1^k \nonumber
			+ 2 \sqrt{\bar{Y}} \sqrt{\bar{\varepsilon}} (L-1) d_1. \nonumber
	\end{align}
	We conclude that for any $\tilde{\xi}_t$ and any constant $\bar{Y} \in \R_{>0}$ with $Y_L(\tilde{\xi}_t) = J_L^*(\tilde{\xi}_t) + W(\tilde{\xi}_t) \leq \bar{Y}$, the value function $J_{L}^*(\tilde{\xi}_t)$ satisfies
	\begin{align} \label{eq:ContLikeOfValue}
		J_{L-1}^*(\tilde{\xi}_{t+1}) \leq J_L^*(\tilde{\xi}_t) - \ell(\bar{u}_0^*(t),\bar{y}_0^*(t)) + \alpha_1(\bar{\varepsilon})
	\end{align}
	with a function $\alpha_1 \in \mathcal{K}_\infty$.\\
	\textbf{Part IV}	
%The central idea for this step of the proof is to use the one-step prediction error $\delta_1^\xi(t)$ \eqref{eq:OneStepPredictionErrorNormBound} to bound the value of $W$ after one step compared to the predicted value and bound the difference w.\,r.\,t. the error.\\
	Since $Y_L(\tilde{\xi}_t) \leq \bar{Y}$, it follows by optimality and the IOSS property \eqref{eq:IOSS} that the nominal prediction satisfies
	\begin{align}
		Y_{L-1}(\hat{\xi}_1^*(t)) \leq Y_L(\tilde{\xi}_t) - \epsilon_{\mathrm{o}} \| \tilde{\xi}(t) \|_2^2 \leq \bar{Y}.
	\end{align}
Given that $Y_L = J_L^*+W$ with $J_L^*$ being nonnegative, we have $W(\hat{\xi}_1^*(t)) \leq \bar{Y}$.
%	By definition of $Y_L$, we have $Y_L = J_L^*+W$ and thus
%	\begin{align}
%		J_{L-1}^*(\bar{\xi}_1^*(t)) + W(\bar{\xi}_1^*(t)) \leq \bar{Y}.
%	\end{align}
%	Since $J_{L-1}^*(\bar{\xi}_1^*(t)) \geq 0$, it follows that $W(\bar{\xi}_1^*(t)) \leq \bar{Y}$.
Using the quadratic nature of $W$ (cf.~\eqref{eq:IOSS}), $W\leq \overline{Y}$, and the one-step prediction error bound~\eqref{eq:OneStepPredictionErrorNormBound}, we get
%	For the next step, this bound on $W(\bar{\xi}_1^*(t))$ is used to derive a bound on $W(\tilde{\xi}_{t+1})$ depending on the one-step prediction error $\delta_1^\xi(t)$. Hence, we use that the IOSS function $W$ is a quadratic form (cf. Lemma \ref{lem:IOSS}) as follows
%	\begin{align}
%			W(\tilde{\xi}_{t+1}) &= W(\bar{\xi}_1^*(t) + \delta_1^\xi(t)) = \| \bar{\xi}_1^*(t) + \delta_1^\xi(t) \|_{P_{\mathrm{o}}}^2\\
%			&\leq \| \bar{\xi}_1^*(t) \|_{P_{\mathrm{o}}}^2 + \| \delta_1^\xi(t) \|_{P_{\mathrm{o}}}^2 + 2 \| \bar{\xi}_1^*(t) \|_{P_{\mathrm{o}}} \cdot \| \delta_1^\xi(t) \|_{P_{\mathrm{o}}}. \nonumber
%	\end{align}
%	Using $\| \bar{\xi}_1^*(t) \|_{P_{\mathrm{o}}}^2 = W(\bar{\xi}_1^*(t))$ and $\| \bar{\xi}_1^*(t) \|_{P_{\mathrm{o}}} = \sqrt{W(\bar{\xi}_1^*(t))} \leq \sqrt{\bar{Y}}$, as well as a bound for weighted norms yields the bound
	\begin{align} \label{eq:ContLikeOfDissip}
		W(\tilde{\xi}_{t+1}) \leq W(\bar{\xi}_1^*(t)) + \alpha_2(\bar{\varepsilon}),
	\end{align}
	where $\alpha_2$ is a class $\mathcal{K}_\infty$-function given by
	\begin{align}\label{eq:IOSS_continuity_alpha}
		\alpha_2(\bar{\varepsilon}) \coloneqq& \lambda_{\max}(P_{\mathrm{o}})(b_\xi(\bar{\varepsilon})+\bar{\varepsilon}\sqrt{pl})^2\\
		& + 2 \sqrt{\lambda_{\max}(P_{\mathrm{o}})} \sqrt{\bar{Y}} (b_\xi(\bar{\varepsilon})+\bar{\varepsilon}\sqrt{pl}).\nonumber
	\end{align}
%	Thus, for any perturbed state $\tilde{\xi}_t$ and any constant $\bar{Y} \in \R_{>0}$ such that Problem \eqref{eq:robustDDUNCMPC} is feasible with $Y_L(\tilde{\xi}_t) = J_L^*(\tilde{\xi}_t) + W(\tilde{\xi}_t) \leq \bar{Y}$, the IOSS function $W(\xi)$ fulfills \eqref{eq:ContLikeOfDissip}.\\
	\textbf{Part V}
	Taking the sum of the bounds \eqref{eq:ContLikeOfValue} and \eqref{eq:ContLikeOfDissip} directly yields the desired inequality~\eqref{eq:ContLikePropOfLyap} with $\alpha_3(\bar{\varepsilon}) \coloneqq \alpha_1(\bar{\varepsilon}) + \alpha_2(\bar{\varepsilon})$.
%\end{proof}
$\hfill\square$

\begin{IEEEbiography}[{\includegraphics[width=1in,height=1.25in,clip,keepaspectratio]{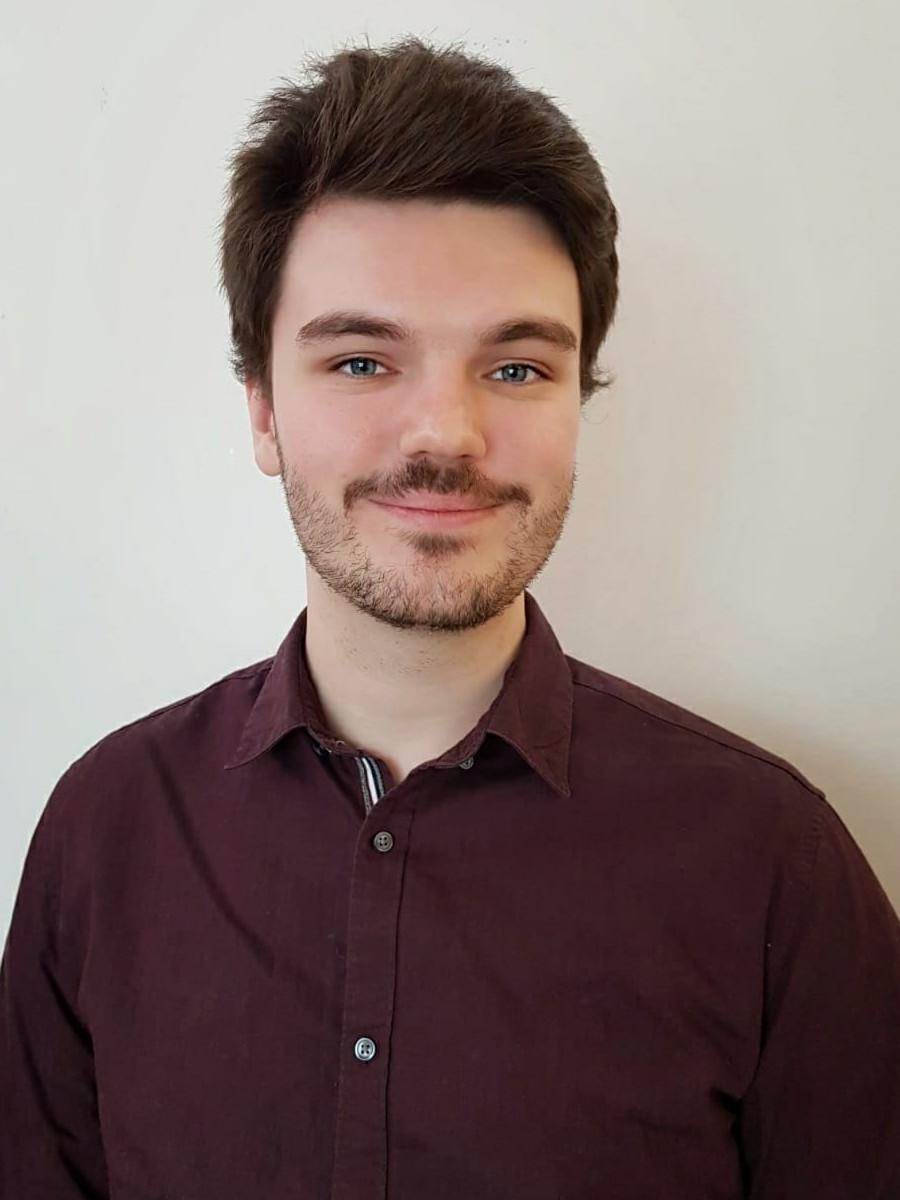}}]{Joscha Bongard}
received his Master degree in Engineering Cybernetics from the University of Stuttgart, Germany, in 2021. 
He is currently pursuing his doctoral studies at the Technical University of Munich (TUM), Germany, at the Chair of Automatic Control.
His research interests include the application of MPC to autonomous driving.
\end{IEEEbiography}

\begin{IEEEbiography}[{\includegraphics[width=1in,height=1.25in,clip,keepaspectratio]{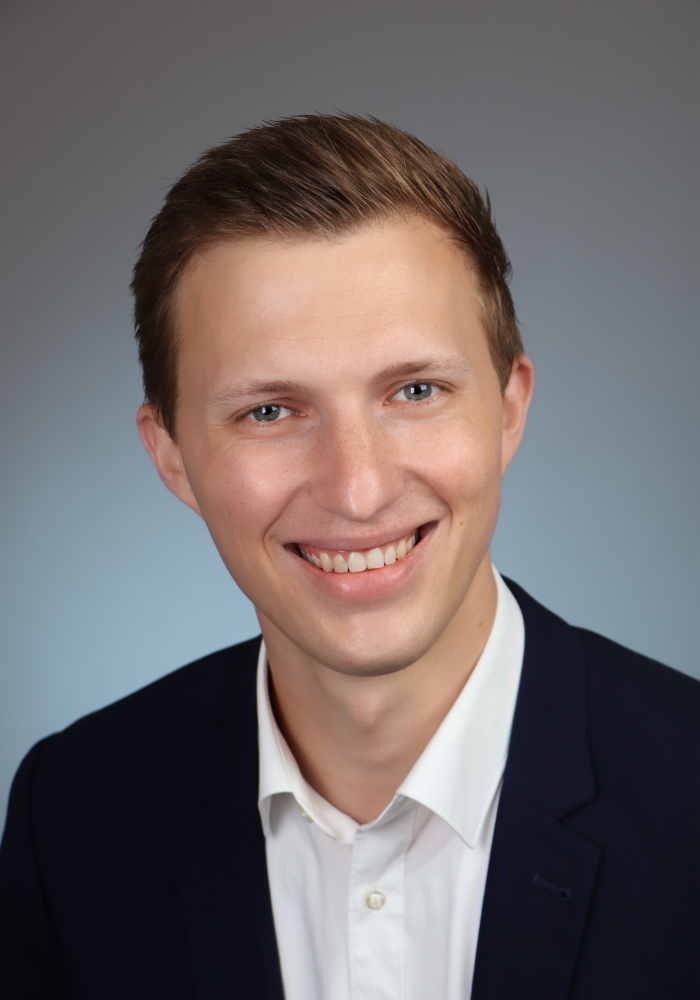}}]{Julian Berberich}
received the Master’s degree in Engineering Cybernetics from the University of Stuttgart, Germany, in 2018. Since 2018, he has been a Ph.D. student at the Institute for Systems Theory and Automatic Control under supervision of Prof. Frank Allg\"ower and a member of the International Max-Planck Research School (IMPRS) at the University of Stuttgart. He has received the Outstanding Student Paper Award at the 59th Conference on Decision and Control in 2020. His research interests are in the area of data-driven analysis and control.
\end{IEEEbiography}

\begin{IEEEbiography}[{\includegraphics[width=1in,height=1.25in,clip,keepaspectratio]{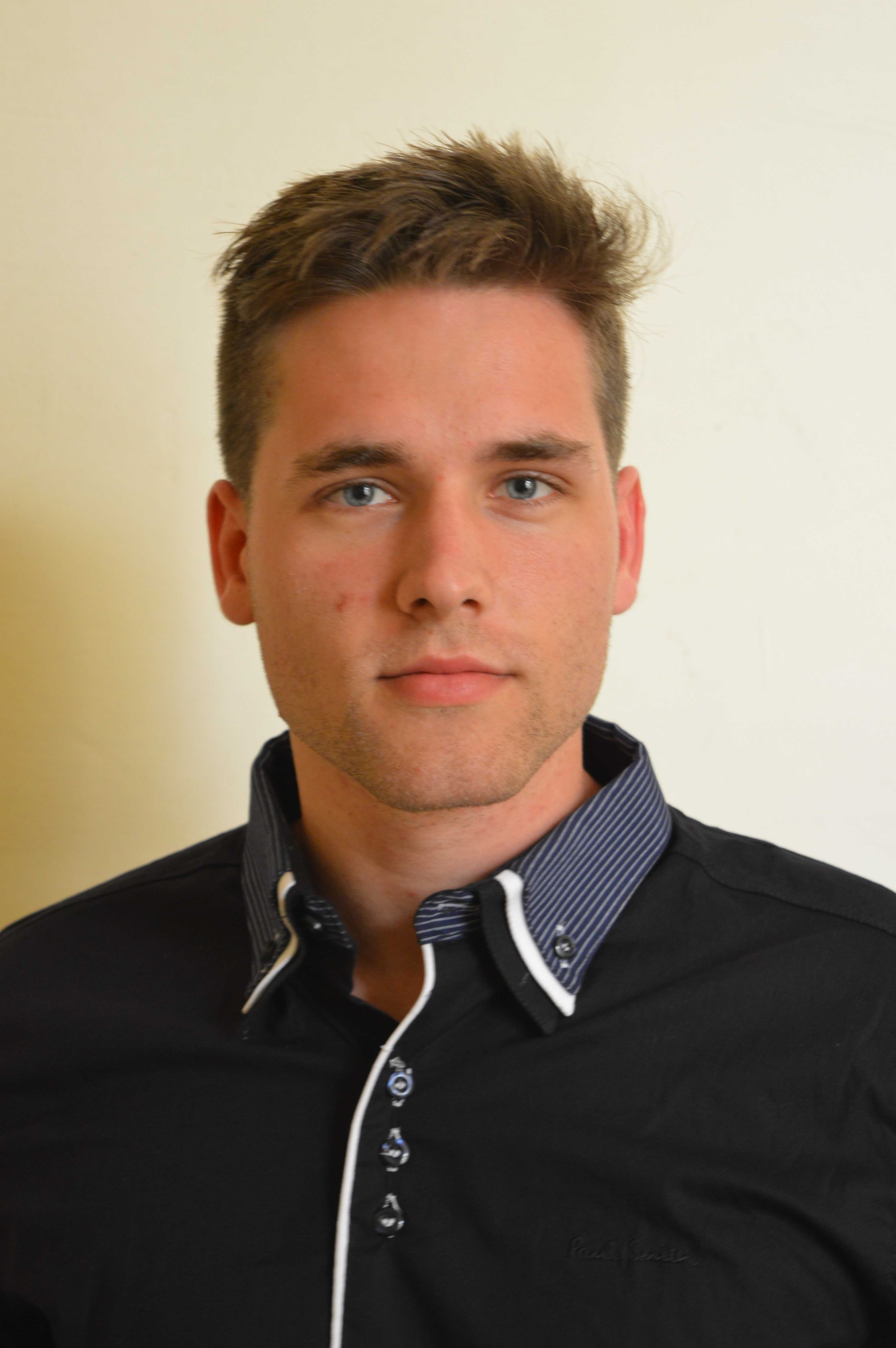}}]{Johannes K\"ohler}
received his Master degree in Engineering Cybernetics from the University of Stuttgart, Germany, in 2017. 
In 2021, he obtained a Ph.D. in mechanical engineering, also from the University of Stuttgart,
Germany.
He is currently a postdoctoral researcher at the Institute for Dynamic Systems and Control (IDSC) at ETH Zürich.
His research interests are in the area of model predictive control and control and estimation for nonlinear uncertain systems. 
 \end{IEEEbiography}

\begin{IEEEbiography}[{\includegraphics[width=1in,height=1.25in,clip,keepaspectratio]{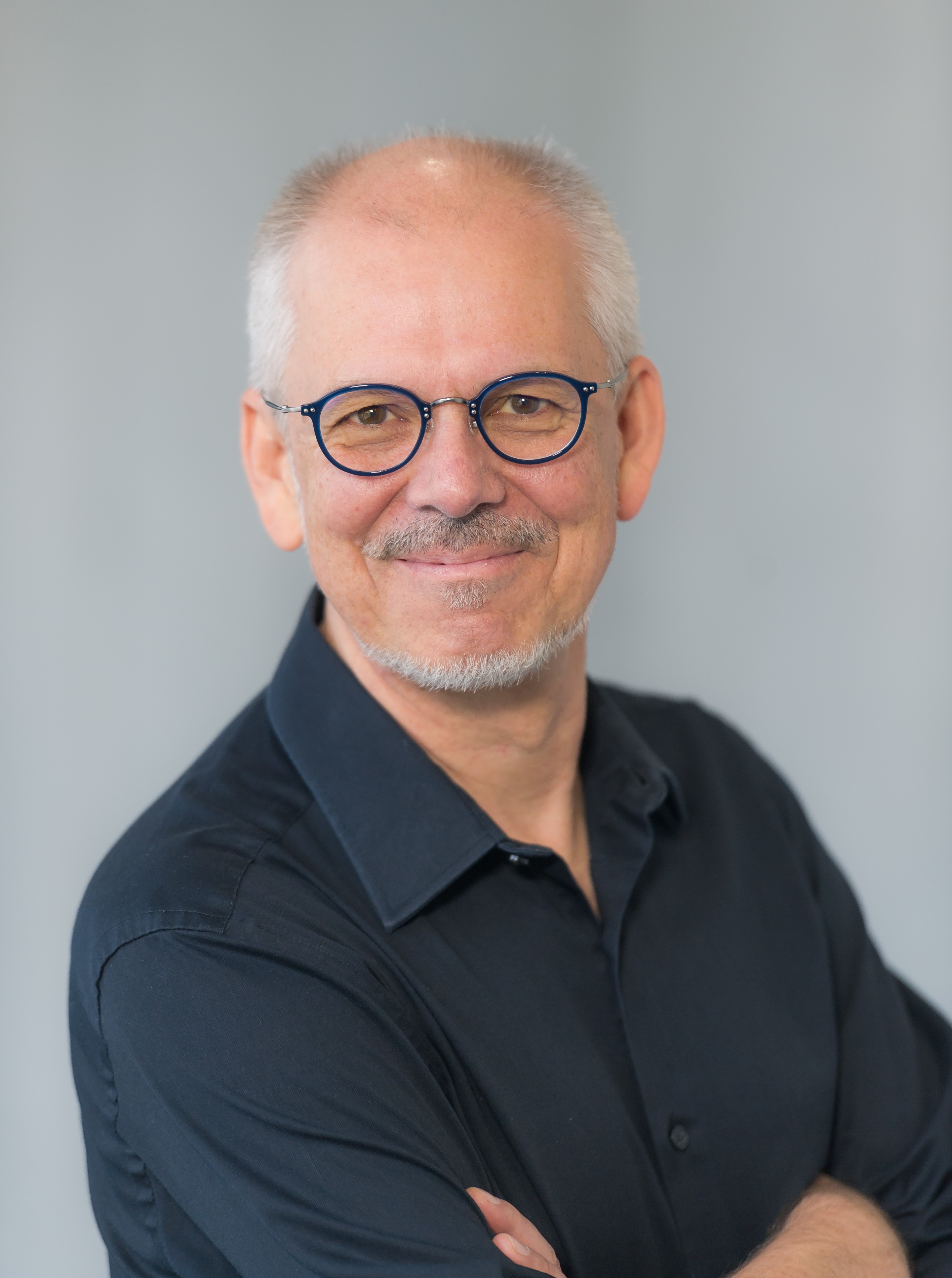}}]{Frank Allg\"ower}
is professor of mechanical engineering at the University of Stuttgart, Germany, and Director of the Institute for Systems Theory and Automatic Control (IST) there.\\ 
Frank is active in serving the community in several roles: Among others he has been President of the International Federation of Automatic Control (IFAC) for the years 2017-2020, Vice-president for Technical Activities of the IEEE Control Systems Society for 2013/14, and Editor of the journal Automatica from 2001 until 2015. From 2012 until 2020 Frank served in addition as Vice-president for the German Research Foundation (DFG), which is Germany’s most important research funding organization. \\
His research interests include predictive control, data-based control, networked control, cooperative control, and nonlinear control with application to a wide range of fields including systems biology.
\end{IEEEbiography}

\end{document}